\documentclass[twocolumn]{article}
\usepackage{graphicx}

\usepackage[ruled,vlined]{algorithm2e}
\usepackage[
  top=2cm,
  bottom=2cm,
  left=1.7cm,
  right=1.7cm
]{geometry}

\usepackage[cmex10]{amsmath}
\usepackage{amsfonts}
\usepackage{amssymb}
\usepackage{mathtools}
\usepackage{nicefrac} 
\usepackage{mathrsfs}
\usepackage{arydshln}
\usepackage{import}
\usepackage{xifthen}
\usepackage{pdfpages}
\usepackage{transparent}
\usepackage{hyperref}
\usepackage{amsthm}

\usepackage{nicematrix}

\DeclareMathOperator{\edgegraph}{edge}
\DeclareMathOperator{\matmap}{M}

\newtheorem{theorem}{Theorem}
\newtheorem{lemma}{Lemma}
\newtheorem{remark}{Remark}
\newtheorem{definition}{Definition}

\newcommand{\cfunof}[1]{\ensuremath{\left\{#1\right\}}}

\newcommand{\funof}[1]{\ensuremath{\left(#1\right)}}
\newcommand{\graphfont}[1]{\ensuremath{\mathcal{#1}}}

\newcommand{\inner}[1]{\ensuremath{\left\langle#1\right\rangle}}
\newcommand{\ltwo}{\ensuremath{\ell_{2+}}}
\newcommand{\N}{\ensuremath{\mathbb{N}}}
\newcommand{\norm}[1]{\ensuremath{\left\Vert #1 \right\Vert}}
\newcommand{\opfont}[1]{\ensuremath{\mathbf{#1}}}
\newcommand{\opfontlower}[1]{\ensuremath{\mathfrak{#1}}}
\newcommand{\R}{\ensuremath{\mathbb{R}}}
\newcommand{\Rring}{\ensuremath{\mathbb{R}_\infty}}
\newcommand{\G}{\ensuremath{\graphfont{G}}}
\newcommand{\shiftop}{\ensuremath{\mathfrak{q}}}
\newcommand{\sqfunof}[1]{\ensuremath{\left[#1\right]}}

\usepackage[nameinlink]{cleveref}
\usepackage{graphicx}
\crefformat{equation}{(#2#1#3)}
\crefmultiformat{equation}{(#2#1#3)}
{ and~(#2#1#3)}{, (#2#1#3)}{ and~(#2#1#3)}
\Crefformat{figure}{#2Fig.~#1#3}
\Crefmultiformat{figure}{Figs.~#2#1#3}{ and~#2#1#3}{, #2#1#3}{ and~#2#1#3}
\usepackage[toc,acronym,nonumberlist]{glossaries}
\usepackage{enumerate,color}

\usepackage{tikz}
\usetikzlibrary{calc, intersections, patterns, through, decorations.pathreplacing,decorations.markings,arrows,shapes,positioning}
\usepackage[american]{circuitikz}

\usepackage{tikz-cd}
\usetikzlibrary{cd}

\usepackage{pgfplots}

\pgfplotsset{compat=newest,width=\columnwidth} 
\pgfplotsset{plot coordinates/math parser=false} 
\newlength\fheight 
\newlength\fwidth 

\usepackage{cleveref}
\usepackage{autonum}
\usepackage{glossaries}
\crefname{problem}{problem}{problems}
\crefname{proposition}{proposition}{propositions}
\crefname{procedure}{procedure}{procedures}
\crefname{assumption}{assumption}{assumptions}
\newacronym{lti}{LTI}{linear time invariant}
\newacronym{lqr}{LQR}{Linear Quadratic Regulator}
\tikzset{every picture/.style=thick}

\definecolor{subsectioncolor}{rgb}{0.067,0.627,0.859}
\definecolor{nblue}{rgb}{0,0.263,0.576}
\definecolor{mblue}{rgb}{0.075,0.541,0.855}
\hypersetup{
    colorlinks = true,
    linkcolor = subsectioncolor,
    citecolor = subsectioncolor,
    urlcolor = subsectioncolor,
    }

\title{Cholesky factorisation, and intrinsically sparse\\ linear quadratic regulation}
\author{Julia Adlercreutz and Richard Pates
\thanks{The authors are with the Department of Automatic Control, Lund University, Box 118, {SE-221 00 Lund}, Sweden.} 
\thanks{The authors are members of the ELLIIT Strategic Research Area at Lund University. This work was supported by the ELLIIT Strategic Research Area. }
}
\date{}

\begin{document}

\maketitle

\begin{abstract}
    We classify a family of matrices of shift operators that can be factorised in a computationally tractable manner
    with the Cholesky algorithm.
    Such matrices arise in the linear quadratic regulator problem, and related areas. We use the factorisation to uncover intrinsic sparsity properties in the control laws for transportation problems with an underlying tree structure. This reveals that the optimal control can be applied in a distributed manner that is obscured by standard solution methods.
\end{abstract}

\section{Introduction}\label{sec:intro}

In this paper we present a sparsity preserving factorisation method for
matrices of shift operators.
The types of matrices we consider arise naturally in the context of optimal control,
such as \gls{lqr} problems.
We study how the factorisation can help reveal sparsity patterns
in the optimal control laws that are otherwise obscured.

Distributed (or decentralised) control aims to find sparse maps $K$ from the measurement sequences $x$ to the control sequences $u$, typically on the form
\begin{equation}\label{eq:map}
    u[k] = K x[k].
\end{equation}
When $K$ is a matrix, as in the equation above, there is a clear connection between the sparsity of the map, and the communication requirements for implementing the control law. In particular if the the $ij$-th entry of $K$ is zero, it means that the $j$-th measurement is not needed to compute the $i$-th control input. A sparse matrix $K$ therefore typically implies that there is less need to communicate measurements when implementing the control. This is particularly important when performing control in the large-scale setting, where it is desirable to rely on only local measurements and conduct design in a manner that is easily scaled. For this reason large networks such as district heating \cite{FW13}, power systems \cite{T16}, and internet congestion \cite{K03} are in practice often controlled in a distributed manner.

The problem of finding optimal controllers under sparsity constraints is difficult in general \cite{BT00}. Quadratic invariance \cite{RL06, LL16} can be used to determine if structural constraints on the map $K$ in \cref{eq:map} can be converted into convex design constraints. However it is often the case that sparsity constraints on \cref{eq:map} are not quadratically invariant. Important cases where quadratic invariance can be applied are explored in \cite{SP10} and \cite{AP22}, where process models with a poset structure play a central role. 

There are also fundamental limitations that are intrinsically coupled with the process of enforcing sparsity constraints on control laws written as in \cref{eq:map}. For instance, \cite{B12} shows that local feedback is not sufficient to suppress large-scale disturbances in low dimensional spatial systems. Furthermore, \cite{LD10} shows that restricting communication can have a negative impact on performance.

In contrast to previous approaches, System Level Synthesis \cite{WMD18,WMD21} imposes sparsity on the closed loop trajectories instead of $K$. This allows for control laws with distributed implementations to be designed without necessarily satisfying sparsity constraints on \cref{eq:map}. In this work we adopt a somewhat similar perspective, however rather than enforcing structure through constraints, we search for inherent structure in the unconstrained optimal control laws. That is, controller structure is not imposed \textit{a priori}. Because of this, the control laws we consider have the same performance and robustness guarantees as the centralised control law (robustness properties of decentralised \gls{lqr} is studied in \cite{KL23}).

\begin{figure}[t]
    \centering
    \begin{tikzpicture}[scale=.3,>=stealth]
	
\coordinate (x1) at (-3, 0);
\draw[thick] ($(x1)+(-0.5,-0.25)$) -- +(0,1) -- +(1,1) -- +(1,0);
\draw[thick] ($(x1)+(-0.65,-0.25)$) -- +(0,1.15) -- +(+0.672,1.8);
\draw[thick] ($(x1)+(+0.65,-0.25)$) -- +(0,1.15) -- +(-0.672,1.8);
\draw[thick] ($(x1)+(-0.5,-0.2225)$) -- ($(x1)+(-0.65,-0.2225)$);
\draw[thick] ($(x1)+(+0.5,-0.2225)$) -- ($(x1)+(+0.65,-0.2225)$);
\draw[thick] ($(x1)+(-0.5,-0.25)$) -- +(0,0.8) -- +(1,0.8);
\draw[thick] ($(x1)+(-0.15,0.65)$) -- ($(x1)+(0.15,0.65)$);

\coordinate (x3) at (-9, 0);
\draw[thick] ($(x3)+(-0.5,-0.25)$) -- +(0,1) -- +(1,1) -- +(1,0);
\draw[thick] ($(x3)+(-0.65,-0.25)$) -- +(0,1.15) -- +(+0.672,1.8);
\draw[thick] ($(x3)+(+0.65,-0.25)$) -- +(0,1.15) -- +(-0.672,1.8);
\draw[thick] ($(x3)+(-0.5,-0.2225)$) -- ($(x3)+(-0.65,-0.2225)$);
\draw[thick] ($(x3)+(+0.5,-0.2225)$) -- ($(x3)+(+0.65,-0.2225)$);
\draw[thick] ($(x3)+(-0.5,-0.25)$) -- +(0,0.8) -- +(1,0.8);
\draw[thick] ($(x3)+(-0.15,0.65)$) -- ($(x3)+(0.15,0.65)$);

\coordinate (x5) at (-15, 0);
\draw[thick] ($(x5)+(-0.5,-0.25)$) -- +(0,1) -- +(1,1) -- +(1,0);
\draw[thick] ($(x5)+(-0.65,-0.25)$) -- +(0,1.15) -- +(+0.672,1.8);
\draw[thick] ($(x5)+(+0.65,-0.25)$) -- +(0,1.15) -- +(-0.672,1.8);
\draw[thick] ($(x5)+(-0.5,-0.2225)$) -- ($(x5)+(-0.65,-0.2225)$);
\draw[thick] ($(x5)+(+0.5,-0.2225)$) -- ($(x5)+(+0.65,-0.2225)$);
\draw[thick] ($(x5)+(-0.5,-0.25)$) -- +(0,0.8) -- +(1,0.8);
\draw[thick] ($(x5)+(-0.15,0.65)$) -- ($(x5)+(0.15,0.65)$);

\coordinate (x7) at (-21, 0);
\draw[thick] ($(x7)+(-0.5,-0.25)$) -- +(0,1) -- +(1,1) -- +(1,0);
\draw[thick] ($(x7)+(-0.65,-0.25)$) -- +(0,1.15) -- +(+0.672,1.8);
\draw[thick] ($(x7)+(+0.65,-0.25)$) -- +(0,1.15) -- +(-0.672,1.8);
\draw[thick] ($(x7)+(-0.5,-0.2225)$) -- ($(x7)+(-0.65,-0.2225)$);
\draw[thick] ($(x7)+(+0.5,-0.2225)$) -- ($(x7)+(+0.65,-0.2225)$);
\draw[thick] ($(x7)+(-0.5,-0.25)$) -- +(0,0.8) -- +(1,0.8);
\draw[thick] ($(x7)+(-0.15,0.65)$) -- ($(x7)+(0.15,0.65)$);

\coordinate (x9) at (-27, 0);
\draw[thick] ($(x9)+(-0.5,-0.25)$) -- +(0,1) -- +(1,1) -- +(1,0);
\draw[thick] ($(x9)+(-0.65,-0.25)$) -- +(0,1.15) -- +(+0.672,1.8);
\draw[thick] ($(x9)+(+0.65,-0.25)$) -- +(0,1.15) -- +(-0.672,1.8);
\draw[thick] ($(x9)+(-0.5,-0.2225)$) -- ($(x9)+(-0.65,-0.2225)$);
\draw[thick] ($(x9)+(+0.5,-0.2225)$) -- ($(x9)+(+0.65,-0.2225)$);
\draw[thick] ($(x9)+(-0.5,-0.25)$) -- +(0,0.8) -- +(1,0.8);
\draw[thick] ($(x9)+(-0.15,0.65)$) -- ($(x9)+(0.15,0.65)$);

\coordinate (x2) at (-6,0);
\draw[fill=black, rounded corners=0.2]
(x2) -- ($ (x2)+(0.2,0) $)
-- ($ (x2)+(0.2,0.2) $)
-- ($ (x2)+(0,0.4) $)
-- ($ (x2)+(-0.2,0.4) $)
-- ($ (x2)+(-0.2,0) $) -- cycle;
\draw[fill=white, rounded corners=0.2]
($ (x2)+(0,0.2) $) --
($ (x2)+(0.175,0.2) $) --
($ (x2)+(0,0.375) $) -- cycle;
\draw[fill=black]
($ (x2)+(-0.25,0) $) --
($ (x2)+(-0.25,0.5) $) --
($ (x2)+(-0.95,0.5) $) --
($ (x2)+(-0.95,0) $) -- cycle;
\draw[white,fill=white] ($ (x2)+(0,0) $) circle (0.1cm);
\draw[fill=black] ($ (x2)+(0,0) $) circle (0.07cm);
\draw[white,fill=white] ($ (x2)+(-0.4,0) $) circle (0.1cm);
\draw[fill=black] ($ (x2)+(-0.4,0) $) circle (0.07cm);
\draw[white,fill=white] ($ (x2)+(-0.8,0) $) circle (0.1cm);
\draw[fill=black] ($ (x2)+(-0.8,0) $) circle (0.07cm);

\coordinate (x4) at (-12,0);
\draw[fill=black, rounded corners=0.2]
(x4) -- ($ (x4)+(0.2,0) $)
-- ($ (x4)+(0.2,0.2) $)
-- ($ (x4)+(0,0.4) $)
-- ($ (x4)+(-0.2,0.4) $)
-- ($ (x4)+(-0.2,0) $) -- cycle;
\draw[fill=white, rounded corners=0.2]
($ (x4)+(0,0.2) $) --
($ (x4)+(0.175,0.2) $) --
($ (x4)+(0,0.375) $) -- cycle;
\draw[fill=black]
($ (x4)+(-0.25,0) $) --
($ (x4)+(-0.25,0.5) $) --
($ (x4)+(-0.95,0.5) $) --
($ (x4)+(-0.95,0) $) -- cycle;
\draw[white,fill=white] ($ (x4)+(0,0) $) circle (0.1cm);
\draw[fill=black] ($ (x4)+(0,0) $) circle (0.07cm);
\draw[white,fill=white] ($ (x4)+(-0.4,0) $) circle (0.1cm);
\draw[fill=black] ($ (x4)+(-0.4,0) $) circle (0.07cm);
\draw[white,fill=white] ($ (x4)+(-0.8,0) $) circle (0.1cm);
\draw[fill=black] ($ (x4)+(-0.8,0) $) circle (0.07cm);

\coordinate (x6) at (-18,0);
\draw[fill=black, rounded corners=0.2]
(x6) -- ($ (x6)+(0.2,0) $)
-- ($ (x6)+(0.2,0.2) $)
-- ($ (x6)+(0,0.4) $)
-- ($ (x6)+(-0.2,0.4) $)
-- ($ (x6)+(-0.2,0) $) -- cycle;
\draw[fill=white, rounded corners=0.2]
($ (x6)+(0,0.2) $) --
($ (x6)+(0.175,0.2) $) --
($ (x6)+(0,0.375) $) -- cycle;
\draw[fill=black]
($ (x6)+(-0.25,0) $) --
($ (x6)+(-0.25,0.5) $) --
($ (x6)+(-0.95,0.5) $) --
($ (x6)+(-0.95,0) $) -- cycle;
\draw[white,fill=white] ($ (x6)+(0,0) $) circle (0.1cm);
\draw[fill=black] ($ (x6)+(0,0) $) circle (0.07cm);
\draw[white,fill=white] ($ (x6)+(-0.4,0) $) circle (0.1cm);
\draw[fill=black] ($ (x6)+(-0.4,0) $) circle (0.07cm);
\draw[white,fill=white] ($ (x6)+(-0.8,0) $) circle (0.1cm);
\draw[fill=black] ($ (x6)+(-0.8,0) $) circle (0.07cm);

\coordinate (x8) at (-24,0);
\draw[fill=black, rounded corners=0.2]
(x8) -- ($ (x8)+(0.2,0) $)
-- ($ (x8)+(0.2,0.2) $)
-- ($ (x8)+(0,0.4) $)
-- ($ (x8)+(-0.2,0.4) $)
-- ($ (x8)+(-0.2,0) $) -- cycle;
\draw[fill=white, rounded corners=0.2]
($ (x8)+(0,0.2) $) --
($ (x8)+(0.175,0.2) $) --
($ (x8)+(0,0.375) $) -- cycle;
\draw[fill=black]
($ (x8)+(-0.25,0) $) --
($ (x8)+(-0.25,0.5) $) --
($ (x8)+(-0.95,0.5) $) --
($ (x8)+(-0.95,0) $) -- cycle;
\draw[white,fill=white] ($ (x8)+(0,0) $) circle (0.1cm);
\draw[fill=black] ($ (x8)+(0,0) $) circle (0.07cm);
\draw[white,fill=white] ($ (x8)+(-0.4,0) $) circle (0.1cm);
\draw[fill=black] ($ (x8)+(-0.4,0) $) circle (0.07cm);
\draw[white,fill=white] ($ (x8)+(-0.8,0) $) circle (0.1cm);
\draw[fill=black] ($ (x8)+(-0.8,0) $) circle (0.07cm);

\coordinate (x1_below) at ($(x1)+(0,-3)$);
\coordinate (x2_below) at ($(x2)+(0,-3)$);
\coordinate (x3_below) at ($(x3)+(0,-3)$);
\coordinate (x4_below) at ($(x4)+(0,-3)$);
\coordinate (x5_below) at ($(x5)+(0,-3)$);
\coordinate (x6_below) at ($(x6)+(0,-3)$);
\coordinate (x7_below) at ($(x7)+(0,-3)$);
\coordinate (x8_below) at ($(x8)+(0,-3)$);
\coordinate (x9_below) at ($(x9)+(0,-3)$);

\draw[->, thick] ($(x9_below)+(0.3,0)$) -- node[midway, mblue, above]
{$u_4$} ($(x8_below)+(-0.3,0)$);
\draw[->, thick] ($(x8_below)+(0.3,0)$) -- node[midway, black, above]
{$1$} ($(x7_below)+(-0.3,0)$);
\draw[->, thick] ($(x7_below)+(0.3,0)$) -- node[midway, mblue, above]
{$u_3$} ($(x6_below)+(-0.3,0)$);
\draw[->, thick] ($(x6_below)+(0.3,0)$) -- node[midway, black, above]
{$1$} ($(x5_below)+(-0.3,0)$);
\draw[->, thick] ($(x5_below)+(0.3,0)$) -- node[midway, mblue, above]
{$u_2$} ($(x4_below)+(-0.3,0)$);
\draw[->, thick] ($(x4_below)+(0.3,0)$) -- node[midway, black, above]
{$1$} ($(x3_below)+(-0.3,0)$);
\draw[->, thick] ($(x3_below)+(0.3,0)$) -- node[midway, mblue, above]
{$u_1$} ($(x2_below)+(-0.3,0)$);
\draw[->, thick] ($(x2_below)+(0.3,0)$) -- node[midway, black, above]
{$1$} ($(x1_below)+(-0.3,0)$);

\node[draw,  thick, fill=white, circle, scale=0.8, label =
$x_1$, name=c1] at (x1_below) {};
\node[draw, thick, fill=white, circle, scale=0.8, label =
$x_2$, name=c2] at (x2_below) {};
\node[draw, thick, fill=white, circle, scale=0.8, label =
$x_3$, name=c3] at (x3_below) {};
\node[draw, thick, fill=white, circle, scale=0.8, label =
$x_4$, name=c4] at (x4_below) {};
\node[draw, thick, fill=white, circle, scale=0.8, label =
$x_5$, name=c5] at (x5_below) {};
\node[draw, thick, fill=white, circle, scale=0.8, label =
$x_6$, name=c6] at (x6_below) {};
\node[draw, thick, fill=white, circle, scale=0.8, label =
$x_7$, name=c7] at (x7_below) {};
\node[draw, thick, fill=white, circle, scale=0.8, label =
$x_8$, name=c8] at (x8_below) {};
\node[draw, thick, fill=white, circle, scale=0.8, label =
$x_9$, name=c9] at (x9_below) {};

\path[->, thick] (c1) edge [out=-45, in=-135, loop below] node[midway,
below] {$1$} (c1);
\path[->, thick] (c3) edge [out=-45, in=-135, loop below] node[midway,
below] {$1-\textcolor{mblue}{u_1}$} (c3);
\path[->, thick] (c5) edge [out=-45, in=-135, loop below] node[midway,
below] {$1-\textcolor{mblue}{u_2}$} (c5);
\path[->, thick] (c7) edge [out=-45, in=-135, loop below] node[midway,
below] {$1-\textcolor{mblue}{u_3}$} (c7);
\path[->, thick] (c9) edge [out=-45, in=-135, loop below] node[midway,
below] {$1-\textcolor{mblue}{u_4}$} (c9);

\end{tikzpicture}
    \caption{Transportation network with a line graph topology.}
    \label{fig:intro-example}
\end{figure}
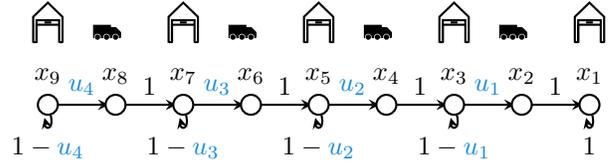

The main contribution of this paper 
is a factorisation for sparse matrices of shift operators. This is then applied
to an \gls{lqr} problem for transportation networks, revealing significant sparsity structure in the underlying control laws.
To illustrate our results, consider the following
\gls{lqr} problem associated with the transportation network in \Cref{fig:intro-example}:
\begin{equation}
    \begin{aligned}
\min_{u\sqfunof{0},u\sqfunof{1},\ldots{}}\sum_{k=0}^\infty{}&
        \frac{1}{2^{k}} \funof{x_1\sqfunof{k}^2   + x_3\sqfunof{k}^2+ \ldots + x_{9}\sqfunof{k}^2}  \\
\text{s.t.}\quad{}
        x_{2i-1}\sqfunof{k+1}   &= x_{2i-1}\sqfunof{k} + x_{2i}\sqfunof{k},   \quad i = 1,\ldots,5,  \\
        x_{2i}\sqfunof{k+1}     &= u_{i}\sqfunof{k}, \quad \quad \quad \quad \quad \quad i = 1,\ldots,4\\
        x\sqfunof{0} &= x_0.
\end{aligned}
\end{equation}
As explained further in \cref{sec:simple-lqr}, the basic objective is to optimally redistributed a commodity through local transportations between neighbouring storage facilities. Solving the optimal control problem above gives a feedback law on the form of \cref{eq:map}, that if written as in \cref{eq:map} would appear to require centralised communication to implement.
However we show that the same control law can factorised in a sparse way as
\begin{equation}\label{eq:fact-lqr-intro}
    \scalebox{0.85}{$
    \setlength\arraycolsep{2pt}
    \underbrace{\begin{bmatrix} \frac{3}{2} & 0 & 0 & 0 \\[3pt]
\tfrac{-1}{2} & \tfrac{7}{6} & 0 & 0 \\[3pt]
0 & \tfrac{-1}{2} & \tfrac{15}{14} & 0 \\[3pt]
0 & 0 & \tfrac{-1}{2} & \tfrac{31}{30}\end{bmatrix}}_{K_1}
$} u[k] = \scalebox{0.85}{$
    \setlength\arraycolsep{2pt}
    \underbrace{ \tfrac{1}{\sqrt{2}}\begin{bmatrix} \tfrac{1}{2} & \tfrac{1}{2} & -1 & -1 & 0 & 0 & 0 & 0 & 0 \\
0 & 0 & \tfrac{1}{2} & \tfrac{1}{2} & -1 & -1  & 0 & 0 & 0 \\
0 & 0 & 0 & 0 & \tfrac{1}{2} & \tfrac{1}{2} & -1 & -1 & 0 \\
0 & 0 & 0 & 0 & 0 & 0 & \tfrac{1}{2} & \tfrac{1}{2} & -1 \end{bmatrix}}_{K_2}
$} x[k].
\end{equation}
That is $K = K_1^{-1} K_2$ is the classical state-feedback solution. This reveals that the $i$-th control input only depends on four state
measurements, and the control input of a neighbouring transportation link. This makes the control law highly amenable to implementations using only local communication.

The method to arrive at the sparse \gls{lqr} solution presented above involves a Cholesky factorisation of
matrices of shift operators. Existence results for the factorisation are presented in \cref{sec:chol}. The associated computational framework and algorithms are given in \cref{sec:comp}, and a software implementation is provided in \cite{A25}. The sparsity properties of the factorisation, and their connections to techniques from spectral factorisation, are discussed in \cref{sec:sparse,sec:spec}. Finally we illustrate how the factorisation can be used to discover sparse factorisations of \gls{lqr} control laws in \cref{sec:control}.

\section*{Notation}\label{sec:notation}

\subsection{Basic notation} $\R$ denotes the field of real numbers, $\R^n$ the vectors with $n$ elements in $\R$, and $\R^{n\times{}m}$ $n\times{}m$ matrices with elements in $\R$. Throughout $\ltwo^n$ denotes the square-summable sequences $y=\funof{y\sqfunof{0} ,
y\sqfunof{1}, \ldots}$, where $y\sqfunof{k} \in \R^n$, with inner product $\inner{x,y} = 
\sum_{k=0}^{\infty} x\sqfunof{k}^\top y\sqfunof{k}  $. Operators on $\ltwo$, and matrices of operators on $\ltwo^n$, are denoted with boldface. The (matrix) identity operator is $\opfont{I}$ (i.e. $\opfont{I}x=x$ for all $x\in\ltwo^n$), and the matrix of zero operators is $\opfont{0}$ (the dimensions will always be clear from context). We write the adjoint of an operator $\opfont{M}$ as $\opfont{M}^*$. $\opfont{M}$ is positive semi-definite (written $\opfont{M}\succeq{}\opfont{0}$) if $\opfont{M}=\opfont{M}^*$, and $\inner{x,\opfont{M}x}\geq{}0$ for all $x\in\ltwo^n$. Where convenient boldface integers denote scaling operators (for example $\opfont{-2}$ is the operator that multiplies every element of a sequence by $-2$). A matrix of operators $\opfont{P}$ is a permutation matrix if it contains exactly one non-zero entry in each row and column, and that entry is equal to $\opfont{1}$. 

\subsection{Shift operators} The forward shift operator is
\begin{equation}\label{eq:forwardshift}
\shiftop:\funof{y[0],y[1],y[2],\ldots{}}\mapsto\funof{y[1],y[2],y[3],\ldots{}}.
\end{equation}
The adjoint of \shiftop{} is the backward shift given by
\[
\shiftop{}^*:\funof{y[0],y[1],y[2],\ldots{}}\mapsto\funof{0,y[0],y[1],\ldots{}}.
\]
Furthermore we define the sets of operators
\begin{equation}\label{eq:multset}
\R[\shiftop,\shiftop^*]=\cfunof{\sum_{i=0}^{n-1}\sum_{j=0}^{n-1}\alpha_{ij}\funof{\shiftop^*}^i\shiftop^j:\alpha_{ij}\in\R},
\end{equation}
and
\begin{equation}\label{eq:rring}
    \R_{\infty} = \cfunof{\sum_{k=0}^n \alpha_k\funof{\shiftop^{*k}\shiftop^k}:n\in\N,\alpha_k\in\R}.
\end{equation}
The algebraic and computational properties of these sets will be motivated and discussed in \cref{sec:comp}.

\section{Results}\label{sec:results}

In this section we study the factorisation of matrices of operators $\opfont{M}^*\opfont{M}$, where the entries of $\opfont{M}$ are in $\R[\shiftop,\shiftop^*]$. As we will see in \cref{sec:control}, such matrices arise in optimal control problems, such as the \gls{lqr} problem. We place particular focus on Cholesky factorisations on the form
\[ \opfont{L}\opfont{L}^*=\funof{\opfont{MP}}^*\opfont{MP},
\]
where $\opfont{L}$ is lower triangular with diagonal entries in $\Rring$, and $\opfont{P}$ is a permutation matrix. These conditions will turn out to be precisely those required to solve least squares problems via back substitution within a polynomially based computational framework. We give necessary and sufficient conditions for the existence of such factorisations in \cref{sec:chol}. We then develop the computational framework and tools that allow the factorisation and back substitution to be performed in \cref{sec:comp}. The sparsity preserving properties of the factorisation are then explored in \cref{sec:sparse}, before we compare our approach to more classically motivated factorisation methods based on spectral factorisation in \cref{sec:spec}.

\subsection{A factorisation for matrices of shift operators}\label{sec:chol}

In this section we characterise factorisability conditions for matrices of operators. Sparsity properties turn out to be of central importance, so we introduce the following set of matrices with sparsity pattern structured by a graph.

\begin{definition}\label{def:mon}
    Let $\mathcal{G}$ be a graph with vertices $\cfunof{v_1,\ldots{},v_n}$, and $m$ underdirected edges $e_i=\cfunof{v_j,v_k}$. Define $\mathbb{M}_{\mathcal{G}}$ to be the set of $n\times{}m$ matrices of operators $\opfont{M}$ such that the $ij$-th entry of $\opfont{M}$ satisfies
    \begin{equation}
        \opfont{M}_{ij} \in \begin{cases}         \cfunof{\alpha_{ij}\shiftop^k,\funof{\shiftop^*}^k\alpha_{ij}}  & \text{if $v_i \in e_j$}, \\
            \cfunof{\opfont{0}} &  \text{otherwise},\\
        \end{cases}
    \end{equation}
    for some $\alpha_{ij}\in\Rring$.
\end{definition}
The following theorem shows that there exists a Cholesky factorisation of $\opfont{M}^*\opfont{M}$ for every $\opfont{M}\in\mathbb{M}_{\mathcal{G}}$, where in addition:
\begin{enumerate}
    \item The diagonal entries of the factor lie in $\Rring$;
    \item The remaining entries of the factor lie in $\R[\shiftop,\shiftop^*]$;
\end{enumerate}
if and only if $\mathcal{G}$ is a tree graph. The significance of this result stems from its computational properties. The existence of a Cholesky factorisation can be established under more general conditions, however ensuring that 1--2) hold will allow the factorisation to be efficiently computed. This will be discussed in \cref{sec:comp}.

\begin{theorem}\label{thm:1}
    Given any undirected graph \graphfont{G}, the following two statements are equivalent:
\begin{enumerate}
    \item[(i)] \graphfont{G} is a tree graph.
    \item[(ii)] For every $\opfont{M}\in\mathbb{M}_{\graphfont{G}}$, there exists a 
        permutation matrix $\opfont{P}$, and a lower triangular matrix $\opfont{L}$ with entries in
        $\R\sqfunof{\shiftop,\shiftop^*}$ and diagonal entries in \Rring{}, such that
    \[
\opfont{L}\opfont{L}^*=\funof{\opfont{MP}}^*\opfont{MP}.
    \]
\end{enumerate}
\end{theorem}

\begin{remark}\label{rem:extend}
    \Cref{thm:1} can also be applied in settings where the entries of $\opfont{M}$ are less constrained. More specifically, provided $\opfont{UMT^{-1}}\in\mathbb{M}_{\mathcal{G}}$ where $\opfont{U}$ is inner ($\opfont{U}^*\opfont{U}=\opfont{I}$), a Cholesky like factorisation
    \[\opfont{L}\opfont{L}^*=\funof{\opfont{MTP}}^*\opfont{MTP}
    \]
    can still be obtained. This will be leveraged in \cref{sec:simple-lqr} when the \gls{lqr} problem is studied.
\end{remark}

\begin{proof}$\text{(i)}\implies{}\text{(ii)}$: The proof will require the
    following facts about operators $\opfontlower{y}\in\Rring$. F3 follows immediately from the definition of $\Rring$, and F1-F2 are proved in \Cref{lem:pseudo} which can be found in \cref{sec:comp}.
\begin{enumerate}
    \item[F1] If $\opfontlower{y}\succeq{}0$, it has a unique positive
        semi-definite square root $\sqrt{\opfontlower{y}}\in\Rring$.
    \item[F2] The Moore-Penrose inverse $\opfontlower{y}^+$ is an operator in $\Rring$.
    \item[F3] $\opfontlower{y}$ is self-adjoint ($\opfontlower{y}=\opfontlower{y}^*$).
\end{enumerate}
Let $m$ denote the number of edges in $\graphfont{G}$. If $m=1$,
$\opfont{M}^*\opfont{M}\in\Rring$, and so $\opfont{M}^*\opfont{M}=
\opfont{PL}\funof{\opfont{PL}}^*$ with
\[
\opfont{P}=\opfont{1}\;\text{and}\;\opfont{L}=\sqrt{\opfont{M}^*\opfont{M}},
\]
which meets all the requirements of (ii) by F1. We now proceed by induction
over $m$, and consider the case that $\graphfont{G}$ has $m+1$ edges. Since
$\graphfont{G}$ is a tree, it must contain an edge that is incident to a
vertex with degree 1, a \textit{leaf edge}. Consequently there exist permutation matrices
$\opfont{P}_{\text{v}}$ and $\opfont{P}_{\text{e}}$ such that for any
$\opfont{M}\in\mathbb{M}_{\graphfont{G}}$,
\[
\opfont{P}_{\text{v}}\opfont{M}\opfont{P}_{\text{e}}=
\begin{bmatrix}
    \overline{\opfont{M}}_{11}&\opfont{0}\\
    \overline{\opfont{M}}_{21}&\overline{\opfont{M}}_{22}\\
    \opfont{0}&\overline{\opfont{M}}_{32}
\end{bmatrix},
\]
where both the blocks $\overline{\opfont{M}}_{11}$ and 
$\overline{\opfont{M}}_{21}$ have dimension 1 by 1. Furthermore
\[
\begin{bmatrix}
    \overline{\opfont{M}}_{22}\\\overline{\opfont{M}}_{32}
\end{bmatrix}\in\mathbb{M}_{\graphfont{G}_1}
\]
where $\graphfont{G}_1$ is a tree graph with $m$ edges. Now define
$\opfont{N} = \opfont{P}_{\mathrm{e}}^* \opfont{M}^* 
\opfont{M}\opfont{P}_{\mathrm{e}}$, and partition $\opfont{N}$ into a 
2 by 2 block matrix with sub-blocks $\opfont{N}_{11}$, $\opfont{N}_{12}$,
$\opfont{N}_{21}$ and $\opfont{N}_{22}$, such that
    \begin{align}
        \!\!\opfont{N}
        \label{eq:blockform}&=
        \begin{bmatrix} \overline{\opfont{M}}_{11}^* \overline{\opfont{M}}_{11} + \overline{\opfont{M}}_{21}^* \overline{\opfont{M}}_{21} & \overline{\opfont{M}}_{21}^*\overline{\opfont{M}}_{22} \\
        \overline{\opfont{M}}_{22}^*\overline{\opfont{M}}_{21} &
        \overline{\opfont{M}}_{22}^* \overline{\opfont{M}}_{22} + \overline{\opfont{M}}_{32}^* \overline{\opfont{M}}_{32}\end{bmatrix}.
    \end{align}
Note in particular that
\begin{enumerate}
    \item $\opfont{N}_{11}\in\Rring$;
    \item $\opfont{N}_{11}\opfont{N}_{11}^+=\opfont{N}_{11}^+\opfont{N}_{11}$;
    \item $\opfont{N}_{11}\opfont{N}_{11}^+\opfont{N}_{12}=\opfont{N}_{12}$.
\end{enumerate}
1)--2) follow immediately from F2--3, and 3) follows since $\opfont{N}\succeq{}\opfont{0}$
as we will now show. With a view to establishing a contradiction, suppose
that 3) does not hold and so $\funof{\opfont{1} - \opfont{N}_{11}
\opfont{N}_{11}^+} \opfont{N}_{12} \neq  0$.
Therefore there exist \ltwo  sequences $x_1$ and $x_2$
such that
\begin{equation}
    \begin{aligned}\label{eq:orthogonal-argument}
    &\inner{x_1, \funof{\opfont{1} - \opfont{N}_{11}\opfont{N}_{11}^+} \opfont{N}_{12} x_2} = \\
        & \qquad=\inner{\funof{\opfont{1} - \opfont{N}_{11} \opfont{N}_{11}^+} x_1, \opfont{N}_{12} x_2 } \neq 0,
    \end{aligned}
\end{equation}
where we have used the fact that $\opfont{1} - \opfont{N}_{11} \opfont{N}_{11}^+$
is self adjoint.
Let $x_3 = \funof{ \opfont{1} - \opfont{N}_{11} \opfont{N}_{11}^+ }
x_1$, and note from equation
\eqref{eq:orthogonal-argument} that $x_3 \neq 0$.
Putting
\[
y=
\begin{bmatrix} \opfont{1} & \opfont{0} \\ -\opfont{N}_{21}\opfont{N}_{11}^+
    & \opfont{I} \end{bmatrix}^*
\begin{bmatrix}
    x_3\\\gamma{}x_2
\end{bmatrix},
\]
consider
\begin{equation}
    \begin{aligned}\label{eq:semidefinite}
        &\inner{y,\opfont{N}y}= \inner{x_3, \opfont{N}_{11} x_3} + \gamma^2 \inner{ x_2, \opfont{N}_{\mathrm{red}}x_2}+\\
        &\quad\qquad\qquad+2\gamma  \inner{x_3, \funof{ \opfont{1}- \opfont{N}_{11}\opfont{N}_{11}^+} \opfont{N}_{12} x_2},
    \end{aligned}
\end{equation}
where
\[
\opfont{N}_{\mathrm{red}} = \opfont{N}_{22} - \opfont{N}_{21} \opfont{N}_{11}^+ \opfont{N}_{12}.
\]
By F3 and properties of the Moore-Penrose inverse,
$\opfont{N}_{11}\funof{ \opfont{1} -\opfont{N}_{11} \opfont{N}_{11}^+}
=\opfont{0}$. It then follows from \cref{eq:semidefinite} that
\[
\inner{y,\opfont{N}y}=\alpha\gamma+\beta\gamma^2
\]
for some $\alpha,\beta\in\R$ with $\alpha\neq{}0$. This implies that there
are values of $\gamma$ such that
$\inner{y,\opfont{N}y}\leq{}0$. However this is a
contradiction since
\[
\opfont{N}\succeq{}0\implies\inner{y,\opfont{N}y}\geq{}0,
\]
and so 3) holds.

Returning  now to \cref{eq:blockform}, direct computation using 1)--3) shows that
\begin{equation}\label{eq:iterative-step}
    \opfont{N} = \begin{bmatrix} \opfont{1} & \opfont{0} \\ \opfont{N}_{21}\opfont{N}_{11}^+ & \opfont{I} \end{bmatrix} 
        \begin{bmatrix} \opfont{N}_{11} & \opfont{0} \\ \opfont{0} & \opfont{N}_{\mathrm{red}} \end{bmatrix} 
        \begin{bmatrix} \opfont{1} & \opfont{N}_{11}^+ \opfont{N}_{12} \\ \opfont{0} & \opfont{I} \end{bmatrix},
\end{equation}
and also that
\[
\opfont{N}_{\mathrm{red}} =\begin{bmatrix}
    \overline{\opfont{M}}_{22}\\\overline{\opfont{M}}_{32}
    \end{bmatrix}^*
    \funof{\opfont{I}-\begin{bmatrix}
        \overline{\opfont{M}}_{21}\\\mathbf{0}
    \end{bmatrix}
    \opfont{N}_{11}^+\begin{bmatrix}
        \overline{\opfont{M}}_{21}\\\mathbf{0}
    \end{bmatrix}^*}
    \begin{bmatrix}
    \overline{\opfont{M}}_{22}\\\overline{\opfont{M}}_{32}
    \end{bmatrix}.
\]
A routine calculation then shows that
$\overline{\opfont{M}}_{21}\opfont{N}^+_{11}\overline{\opfont{M}}_{21}^*\in\Rring$ and also that
\[
\opfont{1}-\overline{\opfont{M}}_{21}\opfont{N}^+_{11}\overline{\opfont{M}}_{21}^*\succeq{}\opfont{0}.
\]
Therefore
\[
\opfont{N}_{\mathrm{red}}=\opfont{M}_{\mathrm{red}}^*\opfont{M}_{\mathrm{red}},
\]
where
\begin{equation}\label{eq:redstep}
\opfont{M}_{\mathrm{red}}=
\begin{bmatrix}
    \sqrt{\opfont{1}-\overline{\opfont{M}}_{21}\opfont{N}^+_{11}\overline{\opfont{M}}_{21}^*}\overline{\opfont{M}}_{22}\\
    \overline{\opfont{M}}_{32}  
\end{bmatrix}\in\mathbb{M}_{\graphfont{G}_1}.
\end{equation}
By the induction hypothesis there exists a permutation matrix $\opfont{P}_{\mathrm{red}}$
and a lower triangular $\opfont{L}_{\mathrm{red}}$ with diagonal entries in $\Rring$ such that
\[
\opfont{N}_{\mathrm{red}}=\opfont{P}_{\mathrm{red}}\opfont{L}_{\mathrm{red}}\funof{\opfont{P}_{\mathrm{red}}\opfont{L}_{\mathrm{red}}}^*.
\]
In light of \cref{eq:iterative-step}, it then follows that that
$\opfont{M}^*\opfont{M}=\opfont{PL}\funof{\opfont{PL}}^*$, with
\begin{equation}\label{eq:factred}
\setlength\arraycolsep{2pt}
\opfont{P} = \opfont{P_e}
\begin{bmatrix}
\opfont{1}&\opfont{0}\\
\opfont{0}&\opfont{P}_{\mathrm{red}}
\end{bmatrix}\;\text{and}\;
\opfont{L}=\scalebox{0.99}{$
\begin{bmatrix}
\sqrt{\opfont{N}_{11}}&\opfont{0}\\
\opfont{P}_{\mathrm{red}}^*\opfont{N}_{21}\sqrt{\opfont{N_{11}}}^+&\opfont{L}_{\mathrm{red}}
\end{bmatrix}$}.
\end{equation}
Hence a suitable factorisation exists for a graph with $m+1$ edges, and the
result follows by induction.

$\neg\text{(i)}\implies\neg\text{(ii)}$:
    Suppose that $\G$ contains a cycle and let 
\begin{equation} \label{eq:cycleM}
        \opfont{M}_{\text{c}} = \begin{bmatrix} -\opfont{1} & \opfont{0} & \ldots & \opfont{0} & \shiftop^* \\
        \shiftop^* & -\opfont{1} & \opfont{0} & \ldots& \opfont{0} \\
        & \ddots &\ddots & \ddots & \\
        \opfont{0} & \opfont{0} & \shiftop^*  & -\opfont{1} & \opfont{0}\\  
        \opfont{0} & \opfont{0} & \opfont{0} & \shiftop^* & -\opfont{1}\end{bmatrix},
    \end{equation}
    where $\opfont{M}_{\text{c}}$ has the same number of rows as there are
    vertices in the cycle. Then there exist permutation matrices
    $\opfont{P}_{\text{v}}$ and $\opfont{P}_{\text{e}}$ such that
    \begin{equation}
        \opfont{M}= 
        \opfont{P}_{\text{v}}^*\begin{bmatrix} \opfont{0} & \opfont{0} \\ \opfont{0} & \opfont{M}_{\text{c}}  \end{bmatrix} 
        \opfont{P}_{\text{e}}^*\in \mathbb{M}_{\G}.
    \end{equation}
    In the remainder of the proof we will show that no factorisation on the
    form in $\text{(ii)}$ exists for this matrix. First observe that $\opfont{M}^*\opfont{M}$ can have such a factorisation if and only if $\opfont{M}_{\text{c}}^* \opfont{M}_{\text{c}} $ does.
    We will now prove that $\opfont{N} = \opfont{M}_{\text{c}}^* \opfont{M}_{\text{c}}$
    has a suitable factorisation only if
    $\opfont{N}_{22} - \opfont{N}_{21} \opfont{N}_{11}^{-1} \opfont{N}_{12} \in \Rring$,
    where $\opfont{N}_{11}$, $\opfont{N}_{12}$, $\opfont{N}_{21}$ and $\opfont{N}_{22}$ are the sub-blocks of a 2 by 2 block matrix partitioning of
    \begin{equation}
        \opfont{M}_{\text{c}}^* \opfont{M}_{\text{c}} =\begin{bmatrix} \opfont{2} & 
            -\shiftop & \opfont{0} &\ldots & \opfont{0} & -\shiftop^*\\
        -\shiftop^* & \opfont{2} & -\shiftop & \opfont{0} & \ldots & \opfont{0} \\
        & \ddots & \ddots & \ddots & & -\shiftop \\
        -\shiftop & \opfont{0} & \ldots & \opfont{0} & -\shiftop^*& \opfont{2}\end{bmatrix},
    \end{equation}
    such that the sub-block $\opfont{N}_{22}$ has dimensions 1 by 1. Let
    \begin{equation}
        \opfont{S} = \begin{bmatrix} \opfont{0} & \opfont{1} \\
        \opfont{I} & \opfont{0}\end{bmatrix}.
    \end{equation}
    When applied from the right, $\opfont{S}$ cycles all the columns one
    step to the left (with the first column becoming the last), and when
    applied from the left cycles all the rows up one (with the first row
    becoming the last). We see that,
    \begin{equation}
        \opfont{N}= \opfont{S}^* \opfont{N}\opfont{S} .
    \end{equation}
    A permutation matrix $\opfont{P}$ has exactly one non-zero matrix element in each row and column. This means that for any permutation matrix, there exists a $k \in \N$ such that
    \begin{equation}
        \opfont{S}^{k} \opfont{P} = \begin{bmatrix} \opfont{P}_{11} & \opfont{0} \\ \opfont{0} &\opfont{1} \end{bmatrix} ,
    \end{equation}
    where $\opfont{P}_{11}$ is a permutation matrix.

    The factorisation in (ii) supposes the existence of a permutation matrix 
    $\opfont{P}$ and lower triangular matrix $\opfont{L}$ with 
    diagonal entries in $\Rring$ such that 
    $\opfont{P}^* \opfont{N}\opfont{P}=\opfont{L}\opfont{L}^*$. However from the above argument we see that there exists a $k \in \N$
    such that
    \begin{equation}\label{eq:Npermuted}
        \setlength\arraycolsep{2pt}
        \opfont{P}^* \opfont{N}\opfont{P}=\opfont{P}^* (\opfont{S}^{k})^* \opfont{N} \opfont{S}^{k} \opfont{P} = \begin{bmatrix} \opfont{P}_{11}^* \opfont{N}_{11} \opfont{P}_{11}
            & \opfont{P}_{11}^* \opfont{N}_{12} \\
            \opfont{N}_{21} \opfont{P}_{11} & \opfont{N}_{22}\end{bmatrix} .
    \end{equation}    
    Any valid factorisation fulfils
    \begin{equation}\label{eq:Nfactorised}
        \begin{split}
            \opfont{P}^* \opfont{N}\opfont{P} &= \begin{bmatrix} \opfont{L}_{11} & \opfont{0} \\ \opfont{L}_{21} & \opfont{L}_{22} \end{bmatrix} 
                \begin{bmatrix} \opfont{L}_{11}^* & \opfont{L}_{21}^* \\ \opfont{0} & \opfont{L}_{22}^* \end{bmatrix} = \\
                    & =\begin{bmatrix} \opfont{L}_{11} \opfont{L}_{11}^* & \opfont{L}_{11}  \opfont{L}_{21}^* \\
                    \opfont{L}_{21}  \opfont{L}_{11}^* & \opfont{L}_{21} \opfont{L}_{21}^* + \opfont{L}_{22} \opfont{L}_{22}^*\end{bmatrix}.
        \end{split}
    \end{equation}
    By comparing the matrices in equations \cref{eq:Npermuted} and \cref{eq:Nfactorised}
    one can see that
    \begin{equation} \label{eq:dnn-expression}
        \begin{split}
            \opfont{L}_{22} \opfont{L}_{22}^*&= \opfont{N}_{22} - 
            \opfont{N}_{21} \opfont{P}_{11} \opfont{P}_{11}^* \opfont{N}_{11}^{-1}
            \opfont{P}_{11} \opfont{P}_{11}^* \opfont{N}_{12}= \\
            &= \opfont{N}_{22} - \opfont{N}_{21} \opfont{N}_{11}^{-1} \opfont{N}_{12}.
        \end{split}
    \end{equation}
    Since $\opfont{L}_{22}$ is a diagonal
    element in $\opfont{L}$, it has to be in $\Rring$ for $\text{(ii)}$ to be fulfilled. This implies that $\opfont{L}_{22} \opfont{L}_{22}^*$
    has to be in $\Rring$, meaning that factorisation on the form in $\text{(ii)}$ exists only if
    $\opfont{N}_{22} - \opfont{N}_{21} \opfont{N}_{11}^{-1} \opfont{N}_{12}\in\Rring$.

    We now computing $\opfont{N}_{22} - \opfont{N}_{21} \opfont{N}_{11}^{-1} \opfont{N}_{12}$
    with the goal of generating a contradiction.
    It can be verified through calculations that
    \begin{equation}
        \opfont{N}_{11}^{-1} = \sum_{k=1}^{n-1} \opfont{N}^{\funof{k}},
    \end{equation}
    where $n$ is the length of the cycle, and the operators
    $\opfont{N}^{\funof{k}}$ have entries equal to
    \begin{equation}
        \opfont{N}^{\funof{k}}_{ij} =
            \frac{(n-i)(n-j)}{k(k+1)} \funof{\shiftop^*} ^{i + k - n}\shiftop^{j + k - n}
    \end{equation}
    whenever $i,j \geq n+k$, with
    $\opfont{N}^{\funof{k}}_{ij} = \opfont{0}$ otherwise. From this one obtains that
    \begin{equation}
    \begin{split}
        &\opfont{N}_{22} - \opfont{N}_{21} \opfont{N}_{11}^{-1} \opfont{N}_{12}=\\ = \frac{n+1}{n}\opfont{1} &- \frac{1}{n} \shiftop ^n -\frac{1}{n} \left( \shiftop^* \right) ^n+
        \sum_{i=1}^{n-1} \frac{1}{i(i+1)}\left( \shiftop^* \right) ^i \shiftop^i.
    \end{split}
    \end{equation}
    This is not in $\Rring$. Hence there does not exist a factorisation fulfilling
    $\text{(ii)}$, and the proof is complete.
\end{proof}

\subsection{Computing with shift operators}\label{sec:comp}

In this section we develop a computational framework for obtaining and working with the factorisation in \Cref{thm:1}. We begin with a discussion of basic algebraic operations involving shift operators in \cref{sec:shiftalg,sec:poly,sec:pseudo}. The key results are then presented in \cref{sec:alg,sec:backsub}, where pseudocode for performing the factorisation is given in \Cref{alg:chol}, and the relevant back substitution computations are discussed. In addition to allowing for the solution of least squares problems such as those presented in \cref{sec:control}, we hope the reader is persuaded that other linear algebra problems and algorithms involving matrices of shift operators can be tackled with our tools. Implementations of all the described operations in both Matlab and Julia are provided in \cite{A25}.

\subsubsection{Shift operator algebra}\label{sec:shiftalg}

The basic building block of the matrices $\opfont{M}$ in \Cref{def:mon} is the forward shift $\shiftop$ given in \cref{eq:forwardshift}. It is important to note that operators that are combinations of $\shiftop$ and $\shiftop^*$ do not commute in general.
For example, $\shiftop \shiftop^* = \opfont{1}$, while
\[
\shiftop^* \shiftop \funof{y\sqfunof{0}, y\sqfunof{1}, y\sqfunof{2}, \ldots}
= \funof{0, y\sqfunof{1}, y\sqfunof{2}, \ldots }.
\]
Aside from this, addition, multiplication and adjoints work as expected. For example
\[
\funof{\shiftop+2\shiftop^*}\funof{\shiftop^*\shiftop+\shiftop}=\shiftop+\shiftop^2+2\shiftop^*\shiftop+2\funof{\shiftop^*}^2\shiftop
\]
and
\[
\funof{\shiftop+\funof{\shiftop^*}^2\shiftop}^* = \shiftop^*+\shiftop^*\funof{\shiftop}^2.
\]
Observe in particular that since $\shiftop \shiftop^* = \opfont{1}$, after performing these operations we always obtain an operator described by a linear combination of terms on the form $\funof{\shiftop^*}^i\shiftop^j$. This is precisely the motivation for the set of operators $\R[\shiftop,\shiftop^*]$ in \cref{eq:multset}. Note that addition, multiplication and adjoints are closed in $\R[\shiftop,\shiftop^*]$. 

\subsubsection{Shift operator computations}\label{sec:poly}

For operators in $\R[\shiftop,\shiftop^*]$, the algebraic operations described above can be performed on a computer by working in terms of the coefficients that describe the corresponding operators. To this end, consider the bijective linear map $\matmap{}:\R[\shiftop,\shiftop^*]\rightarrow{}\R^{n\times{}n}$ defined according to
\[
\sqfunof{\matmap{}\opfontlower{y}}_{ij}\coloneqq{}\alpha_{ij},
\]
where $\alpha_{ij}\in\R$ is the coefficient of the $\funof{\shiftop^*}^i\shiftop^j$ term in $\opfontlower{y}$. That is, $\matmap{}\opfontlower{y}$ returns the matrix of coefficients that specify a given operator $\opfontlower{y}\in\R[\shiftop,\shiftop^*]$. The operations of scaling, addition, multiplication and adjoints of operators $\opfontlower{y},\opfontlower{z}\in\R[\shiftop,\shiftop^*]$ can be equivalently specified in terms of standard operations over $\R$ applied to $\matmap{}\opfontlower{y}$ and $\matmap{}\opfontlower{z}$, as illustrated by the following commutative diagram:

\begin{center}
   \begin{tikzcd}[row sep=huge]
        \R[\shiftop,\shiftop^*]  \arrow[r, "\matmap{}"] \arrow[d, "{+, \times , ^*}"']
        & \R^{n\times n} \arrow[d, "{f_+, f_{\times}, f_*}" ] \\
        \R[\shiftop,\shiftop^*] 
        &  \R^{n\times n} \arrow[l, "\matmap^{-1}" ]
    \end{tikzcd} 
\end{center}

To perform these operators on a computer, we define a data type corresponding to $\matmap{}\opfontlower{y}$. By then implementing the functions $f_+, f_{\times}$ and $f_*$ to perform the addition, multiplications and adjoints for this type, we can perform the corresponding operations for operators in $\R[\shiftop,\shiftop^*]$ on a computer. The details are elementary but messy, so we omit them here, but they can be found in the provided software implementation.

The action of operators in $\R[\shiftop,\shiftop^*]$ on certain sequences in $\ltwo$ can also be described in a computable manner. It is of course necessary that the sequence has a computable description. For our purposes it suffices to consider the sequences defined by triples $\funof{C,A,x_0}$ with $C \in \R^{1 \times n}$, $A \in \R^{n \times n}$, and $x_0 \in \R^n$ according to
\[
\funof{Cx_0,CAx_0,CA^2x_0,\ldots{}}.
\]
The following lemma shows that applying any operator in $\R[\shiftop,\shiftop^*]$ to such a sequence results in a sequence that can also be described by a triple. Hence by introducing a data type in which sequences are represented as triples, computations on sequences can also be performed. As explained in \cref{sec:backsub}, this allows us to apply back substitution to solve equations over $\ltwo$ exactly, and forms the basis of our solutions to \gls{lqr} problems in \cref{sec:control}. 

\begin{lemma}\label{lem:sequence}
    Let $x_0 \in \R^n$, $A \in \R^{n \times n}$, $C \in \R^{1 \times n}$ and $\opfontlower{y} \in \R\sqfunof{\shiftop, \shiftop^*}$. If
    \begin{equation}
        w = \funof{C x_0, C A x_0, C A^2 x_0, \dots}\in\ltwo,
    \end{equation}
    then there exist $\bar{x}_0\in\R^{m}$, $\bar{A}\in\R^{m\times{}m}$, and $\bar{C}^{1\times{}m}$ such that
    \begin{equation}
        \opfontlower{y} w =\funof{\bar{C} \bar{x}_0, \bar{C} \bar{A} \bar{x}_0, \bar{C} \bar{A}^2 \bar{x}_0, \dots}\in\ltwo.
    \end{equation}
\end{lemma}

\begin{proof}
    It is enough to consider $\opfontlower{y} = \shiftop$ and $\opfontlower{y} = \shiftop^*$ as all operators $\opfontlower{y}\in\R[\shiftop,\shiftop^*]$ are linear combinations of these. We start with $\opfontlower{y} = \shiftop$. Then
    \begin{equation}
        \opfontlower{y} w =\funof{C A x_0, C A^2 x_0, C A^3 x_0, \dots}.
    \end{equation}
    From this we can see that $\bar{x}_0 = x_0$, $\bar{A} = A$, and $\bar{C} = C A$.

    We now consider $\opfontlower{y} = \shiftop^*$. Then
    \begin{equation}
        \opfontlower{y} w =\funof{0, C x_0, C A x_0, \dots}.
    \end{equation}
    We see that
    \begin{equation}
        \bar{x}_0 = \begin{bmatrix}
            x_0 \\ 0
        \end{bmatrix}, \quad
        \bar{A} = \begin{bmatrix}
            A & 0 \\ I & 0
        \end{bmatrix}, \quad
        \bar{C} = \begin{bmatrix}
            0 & C
        \end{bmatrix}
    \end{equation}
    gives the correct sequence.
\end{proof}

\subsubsection{Pseudoinverses and square roots}\label{sec:pseudo}

The Cholesky algorithm (and many others), require or benefit from additional operations such as inverses, pseudoinverses and square roots. Unfortunately even when these operations are well defined for $\opfontlower{y}\in\R[\shiftop,\shiftop^*]$, the resulting operators may not be in $\R[\shiftop,\shiftop^*]$. For example, the inverse of $\funof{\opfont{1}-\tfrac{1}{2}\shiftop}$ is a well defined operator on $\ltwo$, however since
\[
\funof{\opfont{1}-\tfrac{1}{2}\shiftop}^{-1}=\sum_{i=0}^\infty{}\funof{\tfrac{1}{2}\shiftop}^i,
\]
it is not an operator in $\R[\shiftop,\shiftop^*]$. This prevents us from additionally computing inverses for every invertible element $\opfontlower{y}\in\R[\shiftop,\shiftop^*]$ through the mapping $\mathrm{M}\, \opfontlower{y}$ in the previous subsection. We can however identify subsets of $\R[\shiftop,\shiftop^*]$ that are closed under the operation of pseudoinverses, and square roots. This is precisely the role of the set of operators $\Rring$ in \cref{eq:rring}. It is easily verified that $\R_{\infty}$ is closed under addition, multiplication and adjoints. The following lemma shows that $\R_\infty$ is also closed under pseudoinverses, and its positive semi-definite elements have positive semi-definite square roots in $\Rring$. 

\begin{lemma}\label{lem:pseudo}
    If $\opfontlower{y} \in \R_{\infty}$, then:
    \begin{enumerate}[(i)]
        \item the Moore-Penrose pseudoinverse $\opfontlower{y}^+$ is in $\R_{\infty}$;
        \item $\opfontlower{y}$ has a unique postive semi-definite square root $\sqrt{\opfontlower{y}}$ if and only if $\opfontlower{y}\succeq{}\opfont{0}$, and furthermore $\sqrt{\opfontlower{y}}\in\R_\infty$.
    \end{enumerate}
\end{lemma}

\begin{proof}
    We will prove both statements by deriving formulas for the pseudoinverse and the square root based on the multiplication of two operators
    \begin{equation}\label{eq:twoop}
    \opfontlower{y} = \sum_{k = 0}^{\infty{}} \alpha_k\funof{\shiftop^{*k} \shiftop^k} \;\;\text{and}\; \;\opfontlower{z} = \sum_{k = 0}^{\infty} \beta_k\funof{\shiftop^{*k} \shiftop^k},
    \end{equation}
    where $\alpha_k=0$ for $k\geq{}n$ and $\beta_k=0$ for $k\geq{}m$.
    Operators in $\R_\infty$ can be equivalently rewritten in terms of the basis operators
    \[
    \opfontlower{e}_k = \shiftop^{*k} \shiftop^k - \shiftop^{* k+1} \shiftop^{k+1}.
    \]
    Denoting the partial sums $\sigma_i = \sum_{k=0}^i \alpha_k$ and $\tau_i = \sum_{k=0}^i 
    \beta_k$ it follows from \cref{eq:twoop} that
    \begin{equation}\label{eq:psum}
        \opfontlower{y} = \sum_{k=0}^{\infty} \sigma_k \opfontlower{e}_k\;\;\text{and}\;\;\opfontlower{z} = \sum_{k=0}^{\infty} \tau_k \opfontlower{e}_k.
    \end{equation}
    This representation is invertible (that is given any $\sigma_k,\tau_k$ such that $\sigma_{k+1}=\sigma_{k}$ and $\tau_{k+1}=\tau_{k}$ for sufficiently large $k$, representations on the form in \cref{eq:psum} can always be rewritten on the form in \cref{eq:twoop}). Now using the fact that $\opfontlower{e}_i^2=\opfontlower{e}_i$ and $\opfontlower{e}_i\opfontlower{e}_j=\opfont{0}$ if $i\neq{}j$, we see that
    \begin{equation}
        \opfontlower{y z} = \sum_{k=0}^{\infty} \sigma_k \tau_k \opfontlower{e}_k.
    \end{equation}
    That is multiplication of operators in $\R_\infty$ corresponds to multiplying their partial sum coefficients. From this observation it follows that $\opfontlower{z}$ is the Moore-Penrose pseudoinverse of $\opfontlower{y}$ if and only if
    \begin{equation}
        \tau_k = \begin{cases}
            1/\sigma_k  \quad &\text{if } ~\sigma_k \neq 0 \\
            0 \quad \quad &\text{otherwise.}
        \end{cases}
    \end{equation}
    Similarly we see that $\opfontlower{y}\succeq{}\opfont{0}$ if and only if $\sigma_k\geq{}0$ for all $k$, and the unique positive semi-definite square root can be obtained by taking the square roots of the partial sum coefficients.
\end{proof}

\subsubsection{The factorisation algorithm}\label{sec:alg}

The key step in finding the factorisation in the proof of \Cref{thm:1} is \cref{eq:redstep}. Based on this, the operators $\opfont{P}$ and $\opfont{L}$ are recursively defined in \cref{eq:factred}. Critically all the calculations in these steps are compatible with the computational framework from \cref{sec:shiftalg,sec:poly,sec:pseudo}. That is, pseudoinverses and square roots are only ever applied to positive semi-definite operators in $\Rring$, and otherwise only additions, multiplications and adjoints are required, all of which are closed in $\R[\shiftop,\shiftop^*]$. This means that we can implement the factorisation outlined in the proof of \Cref{thm:1} in our computational framework. Pseudocode for this procedure is presented in \Cref{alg:chol}.

\begin{algorithm}
\caption{The Cholesky algorithm for $\opfont{L}\opfont{L}^*=\funof{\opfont{MP}}^*\opfont{MP}$}
\label{alg:chol}
\SetKwFunction{Cholesky}{cholesky}
\SetKwFunction{LeafEdgeFirst}{leafEdgeFirstPermutation}
\SetKwFunction{VerteciesFirst}{verteciesFirstPermutation}
\SetKwFunction{FindIndex}{findIndex}
\SetKwProg{Fn}{Function}{:}{end}

\Fn{\Cholesky{$\opfont{M}$}}{

    \If{$\opfont{M}$ has one column}{
        $\opfont{L} = \sqrt{\opfont{M}^* \opfont{M}}$ \;
        $\opfont{P}=\opfont{1}$\;
    }
    \Else{
        $\opfont{P} = \LeafEdgeFirst(\opfont{M})$ \;
        $\opfont{M} = \opfont{MP}$ \;
        $\opfont{Q} = \VerteciesFirst(\opfont{M})$ \hspace{-0.09cm}\;
        \opfont{M}=\opfont{QM}\;      
        $\opfont{M}_{11} = \opfont{M}[1,1]$ \;
        $\opfont{M}_{21} = \opfont{M}[2,1]$ \;
        $\opfont{M}_{22} = \opfont{M}[2,2\!:\!\mathrm{end}]$ \;
        $\opfont{M}_{32} = \opfont{M}[3\!:\!\mathrm{end},2\!:\!\mathrm{end}]$ \;

        $\opfont{N}_{11} = \opfont{M}_{11}^* \opfont{M}_{11} + \opfont{M}_{21}^* \opfont{M}_{21}$ \;
        $\opfont{N}_{21} = \opfont{M}_{22}^* \opfont{M}_{21}$ \;

        $\opfont{M}_{\mathrm{red}} =
        \begin{bmatrix}
            \sqrt{\opfont{1} - \opfont{M}_{21}\opfont{N}_{11}^{+}\opfont{M}_{21}^*}\; \opfont{M}_{22} \\
            \opfont{M}_{32}
        \end{bmatrix}$ \;

        $\opfont{L}_{\mathrm{red}}, \opfont{P}_{\mathrm{red}}
        = \Cholesky(\opfont{M}_{\mathrm{red}})$ \;

        $\opfont{L} =
        \begin{bmatrix}
            \sqrt{\opfont{N}_{11}} & \opfont{0} \\
            \opfont{N}_{21}\sqrt{\opfont{N}_{11}}^{-1} & \opfont{L}_{\mathrm{red}}
        \end{bmatrix}$ \;

        $\opfont{P} =
        \opfont{P}\!
        \begin{bmatrix}
            \opfont{1} & \opfont{0} \\
            \opfont{0} & \opfont{P}_{\mathrm{red}}
        \end{bmatrix}$ \;
    }

    \Return{$\opfont{L}, \opfont{P}$} \;
}
\vspace{0.1cm}
\Fn{\LeafEdgeFirst{$\opfont{M}$}}{

    leafEdgeNotFound = true \;
    c = 0 \;

    \While{leafEdgeNotFound}{
        c = c + 1 \;
        $\opfont{M}_{c} = \opfont{M}[:,c]$ \;

        $[r_1, r_2] = $ Non-zero indices of $\opfont{M}_{c}$ \;

        \If{($\opfont{M}[r_1,2\!:\!\mathrm{end}] = 0$) OR  ($\opfont{M}[r_2,2\!:\!\mathrm{end}] = 0$)}{
            leafEdgeNotFound = false \;
        }
    }

    $\opfont{P} =$ Permutation matrix for swapping column 1 and column c \;
    \Return{$\opfont{P}$} \;
}
\vspace{0.1cm}
\Fn{\VerteciesFirst{$\opfont{M}$}}{

    $\opfont{M}_1 = \opfont{M}[:,1]$ \;
    $[r_1, r_2] =$ Non-zero indices of $\opfont{M}_1$ \;

    $\opfont{P}_1 =$ Permutation matrix for swapping row 1 with row $r_1$ \;
    $\opfont{P}_2 =$ Permutation matrix for swapping row 2 with row $r_2$ \;

    $\opfont{P} = \opfont{P}_1 \opfont{P}_2$ \;

    \Return{$\opfont{P}$} \;
}

\end{algorithm}

\begin{figure}[t]
        \centering
        \begin{tikzpicture}[
        every node/.append style={circle,draw=black, text=black,inner sep=0pt, minimum size=2.5mm}
    ]

    \pgfmathsetmacro{\length}{2}

    \node [label={$v_{1}$}]at (3*\length,0)(v1) { };
    \node [label={$v_{2}$}]at (2*\length,0)(v2) { };
    \node [label={$v_{3}$}]at (\length,0)(v3) { };
    \node [label={$v_{4}$}]at (0,0)(v4) { };
    \path (v2) edge[-] node[below, draw=none] {$e_{1}$} (v1);
    \path (v3) edge[-] node[below, draw=none] {$e_{2}$} (v2);
    \path (v4) edge[-] node[below, draw=none] {$e_{3}$} (v3);

\end{tikzpicture}
        \caption{Graph specifying the sparsity pattern of $\opfont{M}$ in \cref{eq:graph}.}
        \label{fig:example-graph1}
\end{figure}
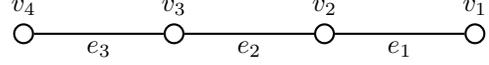

To illustrate the steps in \Cref{alg:chol}, consider the matrix
\begin{equation}\label{eq:graph}
    \opfont{M} = \begin{bmatrix} \shiftop^* & \opfont{0}& \opfont{0} \\
    -\opfont{1} & \shiftop^* & \opfont{0} \\
    \opfont{0} & -\opfont{1} & \shiftop^* \\
    \opfont{0} & \opfont{0} & -\opfont{1}\end{bmatrix} \in \mathbb{M}_{\graphfont{G}},
\end{equation}
where \graphfont{G} is the graph in \Cref{fig:example-graph1}. In the first step of the recursion,  $\opfont{P}^{\text{(1)}}=\opfont{I}$ since the first column of $\opfont{M}$ already corresponds to a leaf edge. Similarly $\opfont{Q}^{\text{(1)}} = \opfont{I}$ since the first column in $\opfont{M}$ has its two non-zero entries in row 1 and row 2. The algorithm then proceeds to compute
\begin{equation}
    \opfont{M}_{\text{red}}^{\text{(1)}} = \begin{bmatrix}
        \sqrt{\opfont{1} - \frac{1}{2}\opfont{1}} \begin{bmatrix} \shiftop^* & 0 \end{bmatrix} \\
        \begin{bmatrix} -\opfont{1} & \shiftop^* \\ \opfont{0} & -\opfont{1} \end{bmatrix} \end{bmatrix}=
            \begin{bmatrix} \frac{1}{\sqrt{2}} \shiftop^* & \opfont{0}\\
            -\opfont{1} & \shiftop^* \\ \opfont{0} & -\opfont{1}\end{bmatrix}.
\end{equation}
The second recursive step then begins as \texttt{cholesky} is called on $\opfont{M}_{\text{red}}^{\text{(1)}}$. In this step $\opfont{P}^{\text{(2)}}=\opfont{I}$ and $\opfont{Q}^{\text{(2)}}= \opfont{I}$ once more, and
\begin{equation}
    \opfont{M}_{\text{red}}^{\text{(2)}} = \begin{bmatrix} \frac{1}{\sqrt{3}} \shiftop^* \\ -\opfont{1} \end{bmatrix}.
\end{equation}
The third recursive step then arrives at the base case, yielding
\begin{equation}
    \opfont{L}_{\text{red}}^{\text{(3)}} = \sqrt{\opfont{M}_{\text{red}}^{\text{(2)}*}\opfont{M}_{\text{red}}^{\text{(2)}}} = \tfrac{2}{\sqrt{3}}\opfont{1}
\end{equation}
and $\opfont{P}_{\text{red}}^{\text{(3)}}=\opfont{1}$. From this $\opfont{P}_{\text{red}}^{\text{(2)}}=\opfont{I}$ and
\begin{equation}
    \opfont{L}_{\text{red}}^{\text{(2)}} = 
    \begin{bmatrix}  \tfrac{\sqrt{3} }{\sqrt{2} }\opfont{1}& \opfont{0} \\
    -\tfrac{\sqrt{2} }{\sqrt{3} } \shiftop &\opfont{L}_{\text{red}}^{\text{(3)}} \end{bmatrix} = \begin{bmatrix}  \tfrac{\sqrt{3} }{\sqrt{2} }\opfont{1}& \opfont{0} \\
    -\tfrac{\sqrt{2} }{\sqrt{3} } \shiftop &\tfrac{2}{\sqrt{3}}\opfont{1} \end{bmatrix},
\end{equation}
and finally $\opfont{P}=\opfont{I}$ and
\begin{equation}\label{eq:L}
    \opfont{L}= \begin{bmatrix} \sqrt{\opfont{2}} & \opfont{0} & \opfont{0} \\
        -\frac{1}{\sqrt{2} }\shiftop & \frac{\sqrt{3}}{\sqrt{2}}  \opfont{1}& \opfont{0} \\
    \opfont{0} & -\frac{\sqrt{2}}{\sqrt{3}}\shiftop & \frac{2}{\sqrt{3}} \opfont{1}\end{bmatrix}
\end{equation}
are computed.

\subsubsection{Back substitution computations}\label{sec:backsub}

The factorisation from \Cref{alg:chol} can be used to solve the system of equations
\begin{equation}\label{eq:linsys}
\opfont{M}^*\opfont{M}z = w
\end{equation}
for $z\in\ltwo^n$ via back substitution whenever $\opfont{M}^*\opfont{M}$ is invertible. In particular the system of equations
\[
\opfont{L}v = \bar{w},
\]
where $\bar{w}=\opfont{P}^*w$, is first solved for each element of $v$ sequentially according to
\begin{equation}\label{eq:backsub}
v_k=\opfont{L}_{kk}^{-1}\funof{\bar{w}_k-\sum_{i = 1}^{k-1}\opfont{L}_{ki}v_i}.
\end{equation}
From this $z$ can be similarly obtained one entry at a time from
\[
\opfont{L}^*\opfont{P}^*z=v.
\]
Since the diagonal entries of $\opfont{L}$ are all in $\Rring$, this can be achieved by applying operators in $\R[\shiftop,\shiftop^*]$ to sequences in $\ltwo$. By \Cref{lem:sequence} it then follows that given any sequence $w$ that can be written on the form
\[
w=\funof{Cx_0,CAx_0,CA^2x_0,\ldots{}},
\]
the triple describing the solution $z$ can be computed using our software tools.

In many applications, it is only necessary to compute the first element of the sequence $z$. We will illustrate this fact in the context of the \gls{lqr} problem in \cref{sec:lqrbacksub}. In this case examining the back substitution steps reveals further simplifications. The following lemma illustrates this fact in a manner tailored specifically to the examples in \cref{sec:control}.

\begin{lemma}\label{thm:control-law}
    Let $\opfont{L} = L_0 + L_1 \shiftop$,
    where $L_0\in\R^{n\times{}n}$ is a lower triangular matrix and $L_1\in\R^{n\times{}n}$ is lower triangular matrix with zeros on the diagonal. If $L_0$ is invertible, then given any $r\in\funof{0,1}$ the solution to
    \begin{equation}
        \opfont{LL}^* z = \funof{r^0K_2x_0,r^1K_2x_0,r^2K_2x_0,\ldots{}}
    \end{equation}
    satisfies
    \begin{equation}
        K_1 z\sqfunof{0} = r^0K_2 x_0,  
    \end{equation}
    where $K_1 = \funof{L_0 + L_1 r} L_0^\top.$
\end{lemma}

\begin{remark}\label{rem:backsubR}
    Observe in particular that in the statement of \Cref{thm:control-law}, $z[0]$ can also be obtained through back substitution (over $\R^n$ rather than $\ltwo^n$), since the factor $K_1$ is given by the product of triangular matrices that are at least as sparse as $\opfont{L}$.
\end{remark}

\begin{remark}
    The above discussions apply to the case that \cref{eq:backsub} has a unique solution for all $v\in\ltwo^n$.  However \Cref{alg:chol} is guaranteed to run for any $\opfont{M}\in\mathbb{M}_\mathcal{G}$, even if $\opfont{M}^*\opfont{M}$ is only positive semi-definite (and hence not invertible). It is easily seen that $\opfont{M}^*\opfont{M}$ is invertible if and only if the diagonal elements of $\opfont{L}$ are invertible, so our software tools can easily detect when the described back substitution procedure applies. Extensions to the semi-definite case can likely be made following standard approaches from linear algebra, but we do not pursue this line further here.
\end{remark}
\begin{proof}
    We start by introducing an intermediate variable $v$ such that
    \begin{equation}
        \begin{cases}
            \opfont{L} v &= \funof{r^0K_2x_0,r^1K_2x_0,r^2K_2x_0,\ldots{}} \\
            \opfont{L}^* z &= v.
        \end{cases}
    \end{equation}
    Because $\opfont{L}^* = L_0^\top + L_1^\top \shiftop^*$ we find that $z\sqfunof{0} = \funof{L_0^\top}^{-1} v\sqfunof{0}. $
    Because of this we only need to compute $v\sqfunof{0}$, not the entire sequence $v$.
    In general we find that
    \begin{equation}
        v\sqfunof{k} + \overline{L} v\sqfunof{k+1} = L_0^{-1} r^kK_2x_0,
    \end{equation}
    where $\overline{L} = L_0^{-1} L_1$. This means that
    \begin{equation}
        v\sqfunof{0} + \funof{-1}^{n+1} \overline{L}^n v\sqfunof{k+n} =
        L_0^{-1} \sum_{i}^{n} \funof{-\overline{L}}^i r^{k+i}K_2x_0.   
    \end{equation}
    Because of the fact that $\overline{L}$
    is nilpotent due to the structure in $L_0$ and $L_1$, we find that
    \begin{equation}
        v\sqfunof{0} =  
        L_0^{-1} \sum_{i}^{\infty} \funof{-\overline{L}r}^i K_2 x\sqfunof{0}.   
    \end{equation}
    The series in the expression above can be recognised as the Neumann series.
    Computation then leads to
    \begin{equation}
        \funof{L_0 + L_1 r} L_0^\top z\sqfunof{0} = K_2 x\sqfunof{0}.
    \end{equation}
\end{proof}

\subsection{Sparse factorisations}\label{sec:sparse}

In this section we show that \Cref{alg:chol} always produces a factorisation that is sparsity preserving in the sense that
\begin{equation}\label{eq:sparsity}
[\funof{\opfont{MP}}^*\opfont{MP}]_{ij}=\opfont{0}\implies{}\opfont{L}_{ij}=\opfont{0}.
\end{equation}
That is, $\opfont{L}$ is at least as sparse as $\funof{\opfont{MP}}^*\opfont{MP}$. This property makes back substitution particularly efficient, and is responsible for the extremely sparse control law factorisations described in the introduction and \cref{sec:control}.

The sparsity preserving property of \Cref{alg:chol} essentially follows from known results based on chordal graphs. We first recall these results. A graph is said to be chordal if each of its cycles of length four or greater has a chord. Consider now a graph with one vertex for each row of the matrix $\opfont{M}^*\opfont{M}$, and an edge between vertices $i$ and $j$ if
\[
[\opfont{M}^*\opfont{M}]_{ij}\neq{}\opfont{0}.
\]
It is well known that if this graph is chordal, then after permuting the rows and columns into a so called perfect elimination ordering, the Cholesky algorithm proceeds without fill in (i.e. if the Cholesky algorithm executes, \cref{eq:sparsity} holds) \cite{KR12}.

We now connect these results to our setting. The only challenge is that \Cref{def:mon} constrains the sparsity pattern of $\opfont{M}$ using a graph, while the classical results for Cholesky factorisation depend on a graph associated with the sparsity pattern of $\opfont{M}^*\opfont{M}$. Consider the following definition, which is illustrated in \Cref{fig:edge-graph}.

\begin{definition}\label{def:edge}
    Let $\edgegraph\funof{\graphfont{G}}$ be the graph with one vertex for
    every edge in $\graphfont{G}$, and edges between two vertices if the two corresponding edges in $\graphfont{G}$ form a path of length two.
\end{definition}

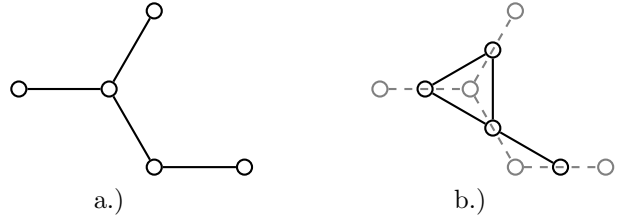
\begin{figure}[t]
    \centering
    \begin{tikzpicture}[scale=0.6,
        every node/.append style={circle,draw=black, text=black,inner sep=0pt, minimum size=2mm}
    ]

    \pgfmathsetmacro{\length}{2}    
    \pgfmathsetmacro{\angle}{120}    
    \pgfmathsetmacro{\distText}{2.5}  
    \pgfmathsetmacro{\distFig}{8}   

    \coordinate (BaseA) at (0,0);
    \node at (BaseA)(v1) { };
    \coordinate (CenterA) at ($(BaseA)+(\length,0)$);
    \node at (CenterA)(v2) { };
    \node  at ($(CenterA) + (\angle/2:\length)$) (v3) {};
    \node  at ($(CenterA) + (-\angle/2:\length)$) (v4) {};
    \node  at ($(CenterA) + (-\angle/2:\length) + (\length,0)$) (v5) {};
    \path (v1) edge[-] node[auto, draw=none] {} (v2);
    \path (v2) edge[-] node[auto, draw=none] {} (v3);
    \path (v2) edge[-] node[auto, draw=none] {} (v4);
    \path (v4) edge[-] node[auto, draw=none] {} (v5);

    \coordinate (TextA) at ($(CenterA)+(0,-\distText)$);
    \node[draw=none] at (TextA) (ta) {a.)};

    \coordinate (BaseB) at (\distFig,0);
    \node [gray] at (BaseB)(v1b) { };
    \coordinate (CenterB) at ($(BaseB)+(\length,0)$);
    \node [gray] at (CenterB)(v2b) { };
    \coordinate (upV) at ($(CenterB) + (\angle/2:\length)$);
    \node [gray]  at (upV) (v3b) {};
    \coordinate (downV) at ($(CenterB) + (-\angle/2:\length)$);
    \node [gray]  at (downV) (v4b) {};
    \coordinate (lastV) at ($(downV)+(\length,0)$);
    \node [gray]  at (lastV) (v5b) {};
    \path (v1b) edge[dashed, gray] node[auto, draw=none] {} (v2b);
    \path (v2b) edge[dashed, gray] node[auto, draw=none] {} (v3b);
    \path (v2b) edge[dashed, gray] node[auto, draw=none] {} (v4b);
    \path (v4b) edge[dashed, gray] node[auto, draw=none] {} (v5b);

    \node at ($0.5*(BaseB)+0.5*(CenterB)$)(e1) {};
    \node at ($0.5*(upV)+0.5*(CenterB)$)(e2) {};
    \node at ($0.5*(downV)+0.5*(CenterB)$)(e3) {};
    \node at ($0.5*(downV)+0.5*(lastV)$)(e4) {};

    \path (e1) edge[-] node[auto, draw=none] {} (e2);
    \path (e1) edge[-] node[auto, draw=none] {} (e3);
    \path (e3) edge[-] node[auto, draw=none] {} (e2);
    \path (e3) edge[-] node[auto, draw=none] {} (e4);

    \coordinate (TextB) at ($(CenterB)+(0,-\distText)$);
    \node[draw=none] at (TextB) (tb) {b.)};

\end{tikzpicture}
    \caption{Illustration of the graphs from \Cref{def:edge}.  a.) shows a graph $\G$, while b.) shows $\G$ in
    grey dashed lines and $\edgegraph\funof{\G}$ in black. }
    \label{fig:edge-graph}
\end{figure}

The relevance of
$\edgegraph\funof{\graphfont{G}}$ is that it describes the sparsity
pattern of $\opfont{M}^*\opfont{M}$ based on the graph $\mathcal{G}$ that characterises the sparsity of $\opfont{M}$ as in \Cref{def:mon}. More specifically, for $i\neq{}j$, $\funof{\opfont{M}^*\opfont{M}}_{ij}\neq{}\opfont{0}$ for some $\opfont{M}\in\mathbb{M}_{\graphfont{G}}$ if and only if vertices $i$ and $j$
are connected in $\edgegraph\funof{\G}$. This shows that after a suitable permutation of the columns of $\opfont{M}$, the Cholesky algorithm is guaranteed to proceed without fill in if $\edgegraph\funof{\graphfont{G}}$ is chordal.

The following theorem precisely characterises chordality properties of $\edgegraph\funof{\G}$ in terms of properties of $\G$. Since tree graphs have no cycles, this result is sufficient to show that all the factorisations characterised by \Cref{thm:1} (and computed by \Cref{alg:chol}) satisfy \cref{eq:sparsity}. 

\begin{theorem}\label{thm:sparsity}
    A graph $\graphfont{G}$ has no cycles of length four or greater if and
    only if $\edgegraph\funof{\graphfont{G}}$ is chordal.
\end{theorem}

\begin{remark}
    The function \texttt{leafEdgeFirstPermuation} in \Cref{alg:chol} provides a permutation matrix such that the rows and columns of $\opfont{MP}^*\opfont{MP}$ are arranged into a perfect elimination ordering.
\end{remark}

\begin{proof}
We prove the result by showing that $\mathcal{G}$ has a cycle of length 4 or more if and only if $\edgegraph\funof{\graphfont{G}}$ is not chordal. Suppose that $\graphfont{G}$ contains a cycle of length $m\geq{}4$. 
Therefore under a suitable labelling of the vertices of $\graphfont{G}$, the set of edges
\[
E_{\text{cycle}}=\cfunof{\cfunof{v_1,v_2},\cfunof{v_2,v_3},\ldots{},\cfunof{v_{m-1},v_m},\cfunof{v_m,v_1}}
\]
form a cycle in $\mathcal{G}$. There is one vertex in 
$\edgegraph\funof{\graphfont{G}}$ for each edge in $E_{\text{cycle}}$. 
Furthermore since there is a bijection between the edges of $\mathcal{G}$
and the vertices of $\edgegraph\funof{\graphfont{G}}$, there is an edge 
between any pair of these vertices in $\edgegraph\funof{\graphfont{G}}$ if
and only if the corresponding edges in $E_{\text{cycle}}$ satisfy 
\[
e_i\cap{}e_j\neq\emptyset{}.
\]
Therefore $\edgegraph\funof{\graphfont{G}}$ contains a cycle of length
$m\geq{}4$ without a chord, and is not chordal. Conversely if
$\edgegraph\funof{\graphfont{G}}$ is not chordal it must contain a cycle 
of length four or greater. Reversing the above arguments then implies that
$\mathcal{G}$ must have a cycle of length four or greater as required.
\end{proof}

\subsection{Comparison with spectral factorisation}\label{sec:spec}

In this section we will make a comparison between the factorisations arising from \Cref{thm:1}, and spectral factorisation. Despite strong similarities, we will see that in general spectral factorisation is not necessarily sparsity preserving in the sense of \cref{eq:sparsity}. This suggests that the non-commutative factorisation framework established in this paper is necessary to fully exploit sparsity in important application areas, such as the \gls{lqr} problem.

The factorisation problem in \Cref{thm:1} is central to the solution of the \gls{lqr} problem. However as we shall see in \cref{sec:control}, in this setting the entries of the matrices $\opfont{M}$ are additionally constrained to lie in
\[
\R[\shiftop^*]=\cfunof{\sum_{j=0}^{n-1}\alpha_j\funof{\shiftop^*}^j:\alpha_j\in\R}.
\]
Clearly $\R[\shiftop^*]\subset\R[\shiftop,\shiftop^*]$, however the operators in $\R[\shiftop^*]$ commute. This fact is exploited by the process of spectral factorisation, which seeks to find an operator $\opfont{L}_{\mathrm{spec}}$ such that,
\[
\opfont{M}^*\opfont{M}=\opfont{L}_{\mathrm{spec}}\opfont{L}_{\mathrm{spec}}^*.
\]
Similarly to the factors from \Cref{thm:1}, the spectral factor $\opfont{L}_{\mathrm{spec}}$ is invertible whenever $\opfont{M}^*\opfont{M}$ is invertible. However the entries in $\opfont{L}_{\mathrm{spec}}$ lie in the $\R[\shiftop]$ (the set of operators with adjoints in $\R[\shiftop^*]$, i.e. those operators that only depend on powers of $\shiftop$), making them simpler to handle computationally, and efficient algorithms for computing $\opfont{L}_{\mathrm{spec}}$ are by now well established.

As we have seen in \cref{eq:L}, sparse spectral factors can be produced by \Cref{alg:chol}. However this is not the case in general. This is best illustrated by example, so consider
\begin{equation}\label{eq:M}
    \opfont{M}=\begin{bmatrix}
        -\opfont{1}&\opfont{0
        }&\opfont{0
        }&\opfont{0
        }\\
        \tfrac{1}{\sqrt{2}}\shiftop^*&-\opfont{1}&\opfont{0
        }&\opfont{0
        }\\
        \opfont{0
        }&\tfrac{1}{\sqrt{2}}\shiftop{}^*&\tfrac{1}{\sqrt{2}}\shiftop{}^*&\opfont{0
        }\\
        \opfont{0
        }&\opfont{0
        }&-\opfont{1}&\tfrac{1}{\sqrt{2}}\shiftop{}^*\\
        \opfont{0
        }&\opfont{0
        }&\opfont{0
        }&-\opfont{1
        }
    \end{bmatrix}.
\end{equation}
As shown in \cref{sec:tvspecfac}, this matrix is connected to a particular \gls{lqr} problem. Computing a factorisation with \Cref{alg:chol} gives $\opfont{P}=\opfont{I}$ and
\[
\opfont{L}=\scalebox{1}{$
    \setlength\arraycolsep{2pt}
\begin{bmatrix}
\opfont{\sqrt{2}}&\opfont{0}&\opfont{0}&\opfont{0}\\-\tfrac{1}{\sqrt{2}}\shiftop^*&\tfrac{\sqrt{3}+2\sqrt{2}}{\sqrt{2}}\opfont{1}-\sqrt{2}\shiftop^*\shiftop&\opfont{0}&\opfont{0}\\
    \opfont{0}&\tfrac{2+\sqrt{3}}{\sqrt{6}}\opfont{1}-\tfrac{1}{\sqrt{2}}\shiftop^*\shiftop&\tfrac{3+2\sqrt{2}}{\sqrt{6}}\opfont{1}-\tfrac{\sqrt{3}}{\sqrt{2}}\shiftop^*\shiftop&\opfont{0}\\
    \opfont{0}&\opfont{0}&-\tfrac{\sqrt{3}}{2}\shiftop{}&\tfrac{\sqrt{5}}{2}\opfont{1}
\end{bmatrix}$}.
\]
This factor is lower triangular and sparsity preserving; the two properties required for efficient back substitution computations. However since $\opfont{L}$ contains $\shiftop^*\shiftop$ terms, it is not a spectral factor. We now prove that in this case no sparsity preserving spectral factor exists, demonstrating that the algebraic framework of \cref{sec:chol,sec:comp} is necessary to find maximally sparse factorisations. Note that the proof is dependent on a technical result concerning the sparsity properties of triangular factorisations, that is presented as \Cref{lem:triangular} in the appendix.
\begin{theorem}\label{thm:specfac}
    If $\opfont{M}$ is as in \cref{eq:M},
    and $\opfont{P}$ is a permutation matrix, then there does not exist an invertible $\opfont{L}$ with entries in $\R[\shiftop]$ such that the following two properties hold:
    \begin{enumerate}
    \item $\opfont{LL}^* = \funof{\opfont{MP}}^*\opfont{MP}$;
        \item \opfont{L} is lower triangular, and satisfies
        \[      [\funof{\opfont{MP}}^*\opfont{MP}]_{ij}=\opfont{0}\implies{}\opfont{L}_{ij}=\opfont{0}.
        \]
    \end{enumerate}
\end{theorem}
\begin{proof}
    The proof will be by contradiction, and so we suppose that an $\opfont{L}$ meeting 1)--2) exists. We will first show that it is only possible to fulfil property 2) for particular permutation matrices, which will yield a contradiction in all but a few special cases. It will then turn out that in these cases, \Cref{alg:chol} produces factors with $\shiftop^*\shiftop$ terms on the diagonal. This will be sufficient to derive a contradiction for all the remaining cases.

    Write $\opfont{L}$ on the form
    \[
    \opfont{L}=\sum_{k}^\infty{}L^{(k)}\shiftop^k,
    \]
    where $L^{(k)}\in\R^{4\times{}4}$. Observe that $L^{(k)}_{ij}\neq{}0$ only if $\opfont{L}_{ij}\neq{}\opfont{0}$, that is the matrices $L^{(k)}$ are at least as sparse as $\opfont{L}$. Let $1^{\funof{n}}\in\ltwo$ be the sequence 
    \[
    1^{\funof{n}}[k]=\begin{cases}
        \frac{1}{\sqrt{n}}&\text{if $k<n$,}\\
        0&\text{otherwise.}
    \end{cases}
    \]
    Since $\opfont{L}$ is bounded, it follows that for any vector $v\in\mathbb{R}^4$,
    \[
    \lim_{n\rightarrow{}\infty}\norm{\opfont{L}^*v1^{(n)}}^2=
    \lim_{n\rightarrow{}\infty}
    \frac{1}{n}\sum_{l=0}^{n-1}\norm{\funof{\sum_{k=0}^{l}L^{(k)}}v}^2<\infty{}.
    \]
    This is sufficient to conclude that
    \[
    L=\sum_{k=0}^\infty{}L^{(k)}
    \]
    exists, and is also at least as sparse as $\opfont{L}$. It also follows that for any vectors $v,w\in\R^{4}$,
    \[
    \lim_{n\rightarrow{}\infty}\inner{w1^{\funof{n}},\opfont{L}\opfont{L}^*v1^{\funof{n}}}=w^\top{}LL^\top{}v.
    \]
    It can similarly be shown that for any vectors $v,w\in\R^{4}$,
        \[
    \lim_{n\rightarrow{}\infty}\inner{\opfont{MP}w1^{\funof{n}},\opfont{MP}v1^{\funof{n}}}=w^\top{}P^\top{}NPv.
    \]
    where $P$ is a permutation matrix and 
    \[
    N=\begin{bmatrix}
        \tfrac{3}{2}&\tfrac{-1}{\sqrt{2}}&0&0\\
        \tfrac{-1}{\sqrt{2}}&\tfrac{3}{2}&\tfrac{1}{2}&0\\
        0&\tfrac{1}{2}&\tfrac{3}{2}&\tfrac{-1}{\sqrt{2}}\\
        0&0&\tfrac{-1}{\sqrt{2}}&\tfrac{3}{2}
    \end{bmatrix}.
    \]
    Hence
    \[
    P^\top{}NP=LL^\top.
    \]
    Now there are 24 possible permutations. It can then be checked that for all but 8 or these, the Cholesky factor of $P^\top{}NP$ is not as sparse as $\opfont{L}$. However it follows from \Cref{lem:triangular} that the Cholesky factor must have the same sparsity as $L$, which is at least as sparse as $\opfont{L}$. This yields a contradiction for all but these 8 permutations. 
    
    Now fix $\opfont{P}$ to be any one of these 8 permutations. It can then be checked that \Cref{alg:chol} runs on $\opfont{MP}$, and let $\opfont{L}_1$ be the resulting factor. Since
    \[
    \opfont{L}_1^{-1}\opfont{L}\funof{\opfont{L}_1^{-1}\opfont{L}}^*=\opfont{I}
    \]
    it follows from \Cref{lem:triangular} that
    \[
    \mathrm{diag}\,\funof{\opfont{L}}\mathrm{diag}\,\funof{\opfont{L}}^*=\mathrm{diag}\,\funof{\opfont{L}_1}\mathrm{diag}\,\funof{\opfont{L}_1}^*.
    \]
    However this again leads to a contradiction, since it can be checked in each case that the right-hand-side of the above contains terms involving $\shiftop{}^*\shiftop{}$, yet the entries of $\opfont{L}$ are assumed in $\R[\shiftop]$ so the left-hand-side cannot.
\end{proof}

\section{Application in control}\label{sec:control}
In this section we will discuss how to apply the factorisation from \cref{sec:results} to give sparsity exploiting controller implementations for some infinite horizon \gls{lqr} problems.
The examples we will consider are simple transportation control problems without penalty on the control sequence.
The simplified setting we consider will allow us to find explicit factorisations of the optimal control on the form
    \begin{equation}\label{eq:control-law}
    K_1 u\sqfunof{k} = K_2 x\sqfunof{k},   
\end{equation}
where $K_1$ and $K_2$ are very sparse and $K_1^{-1}K_2$ is the classic \gls{lqr} state-feedback. 

The hope is to clearly and transparently illustrate many of the interesting aspects of \Cref{thm:1} from the control perspective. The examples we present can certainly be significantly generalised, and the techniques applied in range of related areas (such as filtering and estimation). The rest of the section is structured as follows. In \cref{sec:simple-lqr} we start with a highly simplified example to illustrate the key ideas. This example is inspired by \cite{HPR18}, in which a structured
controller for this problem was found. In \cref{sec:more-general-lqr} we then discuss generalisations where we allow for more complex transportation network topologies.

\begin{remark}
    There are at least three ways to leverage \cref{eq:control-law} to efficiently compute $u[k]$ from the measurements $x[k]$. The first and most closely related to our earlier discussions is to use back substitution (c.f. \Cref{rem:backsubR}). As we will see, the sparsity patterns of $K_1$ and $K_2$ reflect the underlying transportation network topology, allowing for an efficient and often localised implementation based on message passing. Another alternative is to view the computation of $u[k]$ based on $K_1$, $K_2$ and $x[k]$ as a distributed optimisation problem. Again this could be given an efficient implementation based on local communication. Finally the sparsity of $K_1$ and $K_2$ could be used to design market incentives through local shadow prices. These could be used to prompt each storage facility to transport the optimal $u[k]$ based only on local measurements of $x[k]$, and a locally determined market price \cite{HPR20}.
\end{remark}

\subsection{A simple LQR example}\label{sec:simple-lqr} 

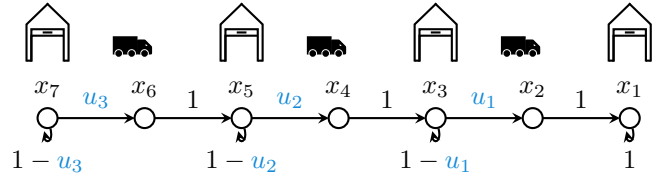
\begin{figure}[t]
    \centering
    \centering 
    \begin{tikzpicture}[scale=.43,>=stealth]

\coordinate (x1) at (-3, 0);
\draw[thick] ($(x1)+(-0.5,-0.25)$) -- +(0,1) -- +(1,1) -- +(1,0);
\draw[thick] ($(x1)+(-0.65,-0.25)$) -- +(0,1.15) -- +(+0.672,1.8);
\draw[thick] ($(x1)+(+0.65,-0.25)$) -- +(0,1.15) -- +(-0.672,1.8);
\draw[thick] ($(x1)+(-0.5,-0.2225)$) -- ($(x1)+(-0.65,-0.2225)$);
\draw[thick] ($(x1)+(+0.5,-0.2225)$) -- ($(x1)+(+0.65,-0.2225)$);
\draw[thick] ($(x1)+(-0.5,-0.25)$) -- +(0,0.8) -- +(1,0.8);
\draw[thick] ($(x1)+(-0.15,0.65)$) -- ($(x1)+(0.15,0.65)$);

\coordinate (x3) at (-9, 0);
\draw[thick] ($(x3)+(-0.5,-0.25)$) -- +(0,1) -- +(1,1) -- +(1,0);
\draw[thick] ($(x3)+(-0.65,-0.25)$) -- +(0,1.15) -- +(+0.672,1.8);
\draw[thick] ($(x3)+(+0.65,-0.25)$) -- +(0,1.15) -- +(-0.672,1.8);
\draw[thick] ($(x3)+(-0.5,-0.2225)$) -- ($(x3)+(-0.65,-0.2225)$);
\draw[thick] ($(x3)+(+0.5,-0.2225)$) -- ($(x3)+(+0.65,-0.2225)$);
\draw[thick] ($(x3)+(-0.5,-0.25)$) -- +(0,0.8) -- +(1,0.8);
\draw[thick] ($(x3)+(-0.15,0.65)$) -- ($(x3)+(0.15,0.65)$);

\coordinate (x5) at (-15, 0);
\draw[thick] ($(x5)+(-0.5,-0.25)$) -- +(0,1) -- +(1,1) -- +(1,0);
\draw[thick] ($(x5)+(-0.65,-0.25)$) -- +(0,1.15) -- +(+0.672,1.8);
\draw[thick] ($(x5)+(+0.65,-0.25)$) -- +(0,1.15) -- +(-0.672,1.8);
\draw[thick] ($(x5)+(-0.5,-0.2225)$) -- ($(x5)+(-0.65,-0.2225)$);
\draw[thick] ($(x5)+(+0.5,-0.2225)$) -- ($(x5)+(+0.65,-0.2225)$);
\draw[thick] ($(x5)+(-0.5,-0.25)$) -- +(0,0.8) -- +(1,0.8);
\draw[thick] ($(x5)+(-0.15,0.65)$) -- ($(x5)+(0.15,0.65)$);

\coordinate (x7) at (-21, 0);
\draw[thick] ($(x7)+(-0.5,-0.25)$) -- +(0,1) -- +(1,1) -- +(1,0);
\draw[thick] ($(x7)+(-0.65,-0.25)$) -- +(0,1.15) -- +(+0.672,1.8);
\draw[thick] ($(x7)+(+0.65,-0.25)$) -- +(0,1.15) -- +(-0.672,1.8);
\draw[thick] ($(x7)+(-0.5,-0.2225)$) -- ($(x7)+(-0.65,-0.2225)$);
\draw[thick] ($(x7)+(+0.5,-0.2225)$) -- ($(x7)+(+0.65,-0.2225)$);
\draw[thick] ($(x7)+(-0.5,-0.25)$) -- +(0,0.8) -- +(1,0.8);
\draw[thick] ($(x7)+(-0.15,0.65)$) -- ($(x7)+(0.15,0.65)$);

\coordinate (x9) at (-27, 0);

\coordinate (x2) at (-6,0);
\draw[fill=black, rounded corners=0.2]
(x2) -- ($ (x2)+(0.2,0) $)
-- ($ (x2)+(0.2,0.2) $)
-- ($ (x2)+(0,0.4) $)
-- ($ (x2)+(-0.2,0.4) $)
-- ($ (x2)+(-0.2,0) $) -- cycle;
\draw[fill=white, rounded corners=0.2]
($ (x2)+(0,0.2) $) --
($ (x2)+(0.175,0.2) $) --
($ (x2)+(0,0.375) $) -- cycle;
\draw[fill=black]
($ (x2)+(-0.25,0) $) --
($ (x2)+(-0.25,0.5) $) --
($ (x2)+(-0.95,0.5) $) --
($ (x2)+(-0.95,0) $) -- cycle;
\draw[white,fill=white] ($ (x2)+(0,0) $) circle (0.1cm);
\draw[fill=black] ($ (x2)+(0,0) $) circle (0.07cm);
\draw[white,fill=white] ($ (x2)+(-0.4,0) $) circle (0.1cm);
\draw[fill=black] ($ (x2)+(-0.4,0) $) circle (0.07cm);
\draw[white,fill=white] ($ (x2)+(-0.8,0) $) circle (0.1cm);
\draw[fill=black] ($ (x2)+(-0.8,0) $) circle (0.07cm);

\coordinate (x4) at (-12,0);
\draw[fill=black, rounded corners=0.2]
(x4) -- ($ (x4)+(0.2,0) $)
-- ($ (x4)+(0.2,0.2) $)
-- ($ (x4)+(0,0.4) $)
-- ($ (x4)+(-0.2,0.4) $)
-- ($ (x4)+(-0.2,0) $) -- cycle;
\draw[fill=white, rounded corners=0.2]
($ (x4)+(0,0.2) $) --
($ (x4)+(0.175,0.2) $) --
($ (x4)+(0,0.375) $) -- cycle;
\draw[fill=black]
($ (x4)+(-0.25,0) $) --
($ (x4)+(-0.25,0.5) $) --
($ (x4)+(-0.95,0.5) $) --
($ (x4)+(-0.95,0) $) -- cycle;
\draw[white,fill=white] ($ (x4)+(0,0) $) circle (0.1cm);
\draw[fill=black] ($ (x4)+(0,0) $) circle (0.07cm);
\draw[white,fill=white] ($ (x4)+(-0.4,0) $) circle (0.1cm);
\draw[fill=black] ($ (x4)+(-0.4,0) $) circle (0.07cm);
\draw[white,fill=white] ($ (x4)+(-0.8,0) $) circle (0.1cm);
\draw[fill=black] ($ (x4)+(-0.8,0) $) circle (0.07cm);

\coordinate (x6) at (-18,0);
\draw[fill=black, rounded corners=0.2]
(x6) -- ($ (x6)+(0.2,0) $)
-- ($ (x6)+(0.2,0.2) $)
-- ($ (x6)+(0,0.4) $)
-- ($ (x6)+(-0.2,0.4) $)
-- ($ (x6)+(-0.2,0) $) -- cycle;
\draw[fill=white, rounded corners=0.2]
($ (x6)+(0,0.2) $) --
($ (x6)+(0.175,0.2) $) --
($ (x6)+(0,0.375) $) -- cycle;
\draw[fill=black]
($ (x6)+(-0.25,0) $) --
($ (x6)+(-0.25,0.5) $) --
($ (x6)+(-0.95,0.5) $) --
($ (x6)+(-0.95,0) $) -- cycle;
\draw[white,fill=white] ($ (x6)+(0,0) $) circle (0.1cm);
\draw[fill=black] ($ (x6)+(0,0) $) circle (0.07cm);
\draw[white,fill=white] ($ (x6)+(-0.4,0) $) circle (0.1cm);
\draw[fill=black] ($ (x6)+(-0.4,0) $) circle (0.07cm);
\draw[white,fill=white] ($ (x6)+(-0.8,0) $) circle (0.1cm);
\draw[fill=black] ($ (x6)+(-0.8,0) $) circle (0.07cm);

\coordinate (x8) at (-24,0);

\coordinate (x1_below) at ($(x1)+(0,-2)$);
\coordinate (x2_below) at ($(x2)+(0,-2)$);
\coordinate (x3_below) at ($(x3)+(0,-2)$);
\coordinate (x4_below) at ($(x4)+(0,-2)$);
\coordinate (x5_below) at ($(x5)+(0,-2)$);
\coordinate (x6_below) at ($(x6)+(0,-2)$);
\coordinate (x7_below) at ($(x7)+(0,-2)$);
\coordinate (x8_below) at ($(x8)+(0,-2)$);
\coordinate (x9_below) at ($(x9)+(0,-2)$);

\draw[->, thick] ($(x7_below)+(0.3,0)$) -- node[midway, mblue, above]
{$u_3$} ($(x6_below)+(-0.3,0)$);
\draw[->, thick] ($(x6_below)+(0.3,0)$) -- node[midway, black, above]
{$1$} ($(x5_below)+(-0.3,0)$);
\draw[->, thick] ($(x5_below)+(0.3,0)$) -- node[midway, mblue, above]
{$u_2$} ($(x4_below)+(-0.3,0)$);
\draw[->, thick] ($(x4_below)+(0.3,0)$) -- node[midway, black, above]
{$1$} ($(x3_below)+(-0.3,0)$);
\draw[->, thick] ($(x3_below)+(0.3,0)$) -- node[midway, mblue, above]
{$u_1$} ($(x2_below)+(-0.3,0)$);
\draw[->, thick] ($(x2_below)+(0.3,0)$) -- node[midway, black, above]
{$1$} ($(x1_below)+(-0.3,0)$);

\node[draw, thick, fill=white, circle, scale=0.8, label =
$x_1$, name=c1] at (x1_below) {};
\node[draw, thick, fill=white, circle, scale=0.8, label =
$x_2$, name=c2] at (x2_below) {};
\node[draw, thick, fill=white, circle, scale=0.8, label =
$x_3$, name=c3] at (x3_below) {};
\node[draw, thick, fill=white, circle, scale=0.8, label =
$x_4$, name=c4] at (x4_below) {};
\node[draw, thick, fill=white, circle, scale=0.8, label =
$x_5$, name=c5] at (x5_below) {};
\node[draw, thick, fill=white, circle, scale=0.8, label =
$x_6$, name=c6] at (x6_below) {};
\node[draw, thick, fill=white, circle, scale=0.8, label =
$x_7$, name=c7] at (x7_below) {};

\path[->, thick] (c1) edge [out=-45, in=-135, loop below] node[midway,
below] {$1$} (c1);
\path[->, thick] (c3) edge [out=-45, in=-135, loop below] node[midway,
below] {$1-\textcolor{mblue}{u_1}$} (c3);
\path[->, thick] (c5) edge [out=-45, in=-135, loop below] node[midway,
below] {$1-\textcolor{mblue}{u_2}$} (c5);
\path[->, thick] (c7) edge [out=-45, in=-135, loop below] node[midway,
below] {$1-\textcolor{mblue}{u_3}$} (c7);

\end{tikzpicture}
    \caption{Model of the transportation network studied in \cref{sec:simple-lqr}.}
    \label{fig:state-graph}
\end{figure}

\subsubsection{Problem setup}\label{sec:lqrsetup} Consider the transportation network illustrated in \Cref{fig:state-graph}, associated with the dynamics
\[
\setlength\arraycolsep{3pt}
\begin{aligned}
x[k+1] &= \underbrace{\begin{bmatrix}
    1&1&0&0&0&0&0\\
    0&0&0&0&0&0&0\\
    0&0&1&1&0&0&0\\
    0&0&0&0&0&0&0\\
    0&0&0&0&1&1&0\\
    0&0&0&0&0&0&0\\
    0&0&0&0&0&0&1\\
\end{bmatrix}}_{\eqqcolon{}A}x[k]+
\underbrace{\begin{bmatrix}
0&0&0\\
1&0&0\\
-1&0&0\\
0&1&0\\
0&-1&0\\
0&0&1\\
0&0&-1\\
\end{bmatrix}}_{\eqqcolon{}B}u[k]\\
y[k]&=\underbrace{\begin{bmatrix}
    1&0&0&0&0&0&0\\
    0&0&1&0&0&0&0\\
    0&0&0&0&1&0&0\\
    0&0&0&0&0&0&1\\
\end{bmatrix}}_{\eqqcolon{}C}x[k].
\end{aligned}
\]
These dynamics model a situation where some quantity (e.g. water, a commodity, etc.) can be transported between a set of storage facilities through a transportation network. The elements of the state vector $x[k]$ with odd indices then specify the amount of the quantity at a given storage facility. In addition each facility has a control sequence that determines how much of the quantity is taken out of the facility, and sent to its neighbour, where it arrives one time step later.

To arrive at an \gls{lqr} problem, suppose that the objective is to optimise transportation about a given operating point. That is, there is some equilibrium flow throughout the network (and all variables are assumed to be defined relative to this equilibrium), and the objective is to adjust the flows to balance the requirements at the individual storage facilities. A natural cost function is
\begin{equation}
    J\funof{x} = \sum_{k=0}^{\infty} r^{2k}
    \funof{x_1\sqfunof{k}^2+x_3\sqfunof{k}^2 + x_5\sqfunof{k}^2 + x_7 \sqfunof{k}^2},
\end{equation}
where $r \in (0,1)$ is a discount factor (since $\sum_ix_i[k]=\sum_ix_i[0]$ for every $k$, a discount factor is required to keep the cost finite for the majority of initial conditions). The optimal inputs can then be computed by solving the \gls{lqr} problem
\begin{equation}
    \begin{aligned}\label{eq:simple-lqr}
\min_{u\sqfunof{0},u\sqfunof{1},\ldots{}}\sum_{k=0}^\infty{}&
        \bar{y}\sqfunof{k}^\top \bar{y}\sqfunof{k}  \\
\text{s.t.}\quad{}
        \bar{x}\sqfunof{k+1}   &= r A \bar{x}\sqfunof{k}  + B u\sqfunof{k},\,\bar{x}[0]=x_0   \\
        \bar{y} &= C \bar{x}\sqfunof{k}. \\
\end{aligned}
\end{equation}
Note that the state $\bar{x}[k+1]$ in the above is related to the original transportation network according to $\bar{x}[k]=r^kx[k]$.

\subsubsection{Least squares reformulation} \label{sec:lqrleastsquares}

The dynamics in \cref{eq:simple-lqr} can be rewritten in terms of shift operators according to
\begin{equation}\label{eq:shift-dynamics}
    \funof{\opfont{1} - r\shiftop^*} \begin{bmatrix} \bar{x}_1  \\ \bar{x}_3 \\
    \bar{x}_5 \\ \bar{x}_7 \end{bmatrix} = \shiftop^* \underbrace{
        \begin{bmatrix} r \shiftop^*  & \opfont{0} & \opfont{0}  \\
        -\opfont{1} & r \shiftop^* & \opfont{0}  \\
        \opfont{0} & -\opfont{1} & r \shiftop^*   \\
        \opfont{0} & \opfont{0} & -\opfont{1}  \end{bmatrix}}_{\opfont{M}}
        \begin{bmatrix} u_1 \\ u_2 \\ u_3 \end{bmatrix} 
            + d,
\end{equation}
where
\begin{equation}\label{eq:init-condition}
    d = \funof{C x[0], C r A x[0],0,0,\ldots{}}.
\end{equation}
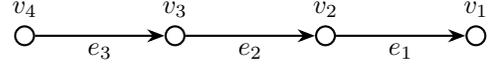
\begin{figure}[t]
        \centering
        \begin{tikzpicture}[
        every node/.append style={circle,draw=black, text=black,inner sep=0pt, minimum size=2.5mm}
    ]

    \pgfmathsetmacro{\length}{2}

    \node [label={$v_{1}$}]at (3*\length,0)(v1) { };
    \node [label={$v_{2}$}]at (2*\length,0)(v2) { };
    \node [label={$v_{3}$}]at (\length,0)(v3) { };
    \node [label={$v_{4}$}]at (0,0)(v4) { };
    \path (v2) edge[-Stealth] node[below, draw=none] {$e_{1}$} (v1);
    \path (v3) edge[-Stealth] node[below, draw=none] {$e_{2}$} (v2);
    \path (v4) edge[-Stealth] node[below, draw=none] {$e_{3}$} (v3);

\end{tikzpicture}
        \caption{Graph representing the transportation network in \Cref{fig:state-graph}. There is one vertex per storage facility, and one edge per transportation link. The edges have been oriented to reflect the transportation direction. The undirected version of this graph (i.e. \Cref{fig:example-graph1}) describes the sparsity pattern of the $\opfont{M}$ appearing in \cref{sec:lqrleastsquares} according to \Cref{def:mon}.}
        \label{fig:example-graph}
\end{figure}
The sparsity pattern of $\opfont{M}$ is captured by the graph in \Cref{fig:example-graph}, and this is the graph that is relevant when applying \Cref{thm:1}. This graph also captures the essential topological features of the transmission network, with one vertex per storage facility, and one edge per one step delay. The sequence $d$  captures the effect of the initial condition. In line with \Cref{rem:extend}, we conceal the storage dynamics in \cref{eq:shift-dynamics} with the following variable transforms
\begin{equation}\label{eq:variable-transform}
    \begin{cases}
        \funof{\opfont{1} - r \shiftop^*} z &= u \\
        \funof{\opfont{1} - r \shiftop^*} y_{\text{init}} &= d.
    \end{cases}
\end{equation}
Substituting these constraints directly into the cost function shows that the least squares problem
\begin{equation}\label{eq:ls-problem}
    \min_{z\in\ltwo}\sum_{k=0}^\infty{} 
    \| \shiftop^* \opfont{M} z + y_{\text{init}}\|^2.
\end{equation}
is equivalent to the \gls{lqr} problem in \cref{eq:simple-lqr}.

\subsubsection{Solution by back substitution}\label{sec:lqrbacksub}

The least squares problem in \cref{eq:ls-problem} is solved by any $z\in\ltwo^n$ that satisfies
\begin{equation}\label{eq:ls-solution}
    \opfont{M}^*\opfont{M} z = - \shiftop \opfont{M}^* y_{\text{init}}.
\end{equation}
An optimal $z$ can be found in a sparsity exploiting manner by making use of the tools from \cref{sec:results}. Applying \Cref{alg:chol} gives $\opfont{P}=\opfont{I}$ and an $\opfont{L}$ such that $\opfont{M}^*\opfont{M}=\opfont{L}^*\opfont{L}$. As an example, with $r = 1/\sqrt{2}$,
\[\label{eq:Lexample}
\opfont{L}=\begin{bmatrix}
    \sqrt{\frac{3}{2}} \opfont{1} & \opfont{0} & \opfont{0} \\
    -\frac{1}{\sqrt{3}} \shiftop & \sqrt{\frac{7}{6}}\opfont{1} & \opfont{0} \\
    \opfont{0} & -\sqrt{\frac{3}{7}} \shiftop & \sqrt{\frac{15}{14}}\opfont{1}
\end{bmatrix}.
\]
An optimal $z$ can then be found by sequentially solving the equations
\[
\opfont{L}v=w,\,\opfont{L}^*z=v
\]
through back substitution as described in \cref{sec:backsub}, where $w= - \shiftop \opfont{M}^* y_{\text{init}}$. 

To proceed we must first represent $w$ as a triple. Writing the operator matrix $\opfont{M}^*$ as $\opfont{M}^*=\funof{M^{\text{(0)}}}^\top + \funof{M^{\text{(1)}}}^\top \shiftop$, we see that $w$ is the
response of the system 
\begin{equation}
\begin{aligned}
        \tilde{x}\sqfunof{k+1}   &= \begin{bmatrix} rA & 0 \\ CrA & 0 \end{bmatrix}
        \tilde{x}\sqfunof{k}, \;\tilde{x}\sqfunof{0} =  \begin{bmatrix}
\bar{x}\sqfunof{0} \\ 0  \end{bmatrix}^\top  \\ \\
        w\sqfunof{k}   &= \begin{bmatrix} \funof{M^{\text{(1)}}}^\top CrA  & \funof{M^{\text{(0)}}}^\top  \end{bmatrix}  \tilde{x}\sqfunof{k}  
\end{aligned}
\end{equation}
This is already on a suitable form, however since $A= A^2$ this representation can be further simplified, and we find that
\begin{equation}\label{eq:w}
    w\sqfunof{k} = r^k \funof{\funof{M^{\text{(0)}}}^\top + \funof{M^{\text{(1)}}}^\top r  } C r A x\sqfunof{0}. 
\end{equation}
That is the sequence $w$ is represented by the triple 
\[
\funof{\funof{\funof{M^{\text{(0)}}}^\top + \funof{M^{\text{(1)}}}^\top r  } C r A, rI, x[0]}.
\]
Now that the sequence is known, back substitution can be performed in a sparsity exploiting manner to obtain $z$ from which $u$ can be found through \cref{eq:variable-transform}. 

\subsubsection{Structured control laws} \label{sec:lqrstructurecontrol}

The above process shows that the entire optimal input sequence $u\in\ltwo^n$ can be computed in way that exploits sparsity in the underlying transportation network. This is rather more than is needed in practice, where it is natural to apply inputs in a receding horizon manner (and for the \gls{lqr} problem the optimal controls are given by a static state-feedback). That is, it is enough to characterise the map
\[
K:x[0]\mapsto{}u[0].
\]
From \cref{eq:Lexample,eq:w} we see that both $\opfont{L}$ and $w$ have the necessary structure to apply \Cref{thm:control-law}. Furthermore since $z\sqfunof{0} = u\sqfunof{0}$ we can arrive at the control law directly without having to invert \cref{eq:variable-transform}. It then follows that when $r = 1/\sqrt{2}$ we get a factorisation of the control law on the form in \cref{eq:control-law} with
\begin{equation}
	K_1 = \begin{bmatrix}
    \frac{3}{2} & 0 & 0\\
    -\frac{1}{2} & \frac{7}{6} & 0 \\
    0 & - \frac{1}{2} & \frac{15}{14}
    \end{bmatrix}
\end{equation}
and
\begin{equation}
	K_2 = \frac{1}{2\sqrt{2}}\begin{bmatrix}
    1 & 1 & -2 & -2 & 0 & 0 & 0 \\
    0 & 0 & 1 & 1 & -2 & -2 & 0 \\
    0 & 0 & 0 & 0 & 1 & 1 & -2
    \end{bmatrix}.
\end{equation}
The structure
of $K_1$ and $K_2$ shows that in order to compute the $i$-th input, only the
state-variables in the neighbouring vertices and knowledge
of the control sequence along one of the neighbouring edges is required.

\begin{figure*}
    \hspace{-1.5cm}
    \begin{tikzpicture}[
		every node/.append style={circle,draw=black, text=black,inner sep=0pt, minimum size=2.5mm, scale=0.95}
	]

    \pgfmathsetmacro{\distGraph}{8.5}  
    \pgfmathsetmacro{\distEq}{5}  
    \pgfmathsetmacro{\EqCenter}{1.5}  

	\node [label={$y_{1}$}]at (-5.815785e+00,4.867149e+00)(y1) { };
	\node [label={$y_{2}$}]at (-4.190486e+00,2.018516e+00)(y2) { };
	\node [label={$y_{3}$}]at (-1.911586e+00,3.629772e+00)(y3) { };
	\node [label={$y_{4}$}]at (-4.496418e+00,3.704133e+00)(y4) { };
	\node [label={$y_{5}$}]at (-3.099428e+00,3.675670e+00)(y5) { };
	\node [label=below:{$y_{6}$}]at (1.130498e+00,-4.352838e+00)(y6) { };
	\node [label=below:{$y_{7}$}]at (-1.126788e-01,-4.696911e+00)(y7) { };
	\node [label={$y_{8}$}]at (-5.388307e+00,-4.898947e+00)(y8) { };
	\node [label={$y_{9}$}]at (-4.147545e+00,-3.747024e+00)(y9) { };
	\node [label=left:{$y_{10}$}]at (-1.729980e+00,-6.679899e-01)(y10) { };
	\node [label=right:{$y_{11}$}]at (-2.616373e+00,2.140913e+00)(y11) { };
	\node [label=right:{$y_{12}$}]at (1.161214e-01,-2.984829e+00)(y12) { };
	\node [label={$y_{13}$}]at (-2.499658e+00,-2.249805e+00)(y13) { };
	\node [label={$y_{14}$}]at (2.639575e-01,5.359162e-01)(y14) { };
	\node [label={$y_{15}$}]at (8.915595e+00,2.103629e+00)(y15) { };
	\node [label={$y_{16}$}]at (7.374628e+00,1.518869e+00)(y16) { };
	\node [label=below:{$y_{17}$}]at (6.835995e+00,2.208639e-02)(y17) { };
	\node [label={$y_{18}$}]at (-5.442719e-01,-4.833373e-01)(y18) { };
	\node [label=below:{$y_{19}$}]at (3.876326e+00,-6.793099e-01)(y19) { };
	\node [label={$y_{20}$}]at (5.395267e+00,5.908330e-01)(y20) { };
	\node [label={$y_{21}$}]at (2.644130e+00,-4.649480e-02)(y21) { };
	\path (y4) edge[-Stealth] node[auto, draw=none] {$u_{1}$} (y1);
	\path (y11) edge[-Stealth] node[auto, draw=none] {$u_{2}$} (y2);
	\path (y11) edge[-Stealth] node[auto, draw=none] {$u_{3}$} (y3);
	\path (y11) edge[-Stealth] node[auto, draw=none] {$u_{4}$} (y4);
	\path (y11) edge[-Stealth] node[auto, draw=none] {$u_{5}$} (y5);
	\path (y12) edge[-Stealth] node[auto, draw=none] {$u_{6}$} (y6);
	\path (y12) edge[-Stealth] node[left, draw=none] {$u_{7}$} (y7);
	\path (y9) edge[-Stealth] node[auto, draw=none] {$u_{8}$} (y8);
	\path (y13) edge[-Stealth] node[auto, draw=none] {$u_{9}$} (y9);
	\path (y18) edge[-Stealth] node[above, draw=none] {$u_{10}$} (y10); 
	\path (y18) edge[-Stealth] node[auto, draw=none] {$u_{11}$} (y11);
	\path (y18) edge[-Stealth] node[auto, draw=none] {$u_{12}$} (y12);
	\path (y18) edge[-Stealth] node[auto, draw=none] {$u_{13}$} (y13);
	\path (y18) edge[-Stealth] node[auto, draw=none] {$u_{14}$} (y14);
	\path (y16) edge[-Stealth] node[auto, draw=none] {$u_{15}$} (y15);
	\path (y20) edge[-Stealth] node[auto, draw=none] {$u_{16}$} (y16);
	\path (y20) edge[-Stealth] node[auto, draw=none] {$u_{17}$} (y17);
	\path (y21) edge[-Stealth] node[auto, draw=none] {$u_{18}$} (y18);
	\path (y21) edge[-Stealth] node[auto, draw=none] {$u_{19}$} (y19);
	\path (y21) edge[-Stealth] node[auto, draw=none] {$u_{20}$} (y20);

    \tiny
    \setcounter{MaxMatrixCols}{50}
    \scalebox{0.9}{
        \node [draw=none] at (\EqCenter, -\distGraph)(K1) {$\input{big-tree-K1.tex}$};
        \node [draw=none] at (\EqCenter, -\distGraph-\distEq)(K2) {$\input{big-tree-K2.tex}$};
    }
\end{tikzpicture}

    \vspace{-9.9cm}
    \caption{Illustration of the sparsity of the control law in \cref{eq:control-law} for a larger transportation network with a tree structure. The graph represents the network topology, with one vertex per storage facility, and one directed edge per transportation link. The symbol $\boldsymbol{\star}$ represents the non-zero entries in the matrices that describe the control law. The sparsity pattern in the matrices corresponds closely to the topology of the graph. For example, consider the top left corner of the graph where the control sequence $u_1$ is located. Since it corresponds to a leaf edge, it does not depend on any other control sequence. All that is needed are the measurements $y_1\sqfunof{k}$, $u_1\sqfunof{k-1}$, $y_4\sqfunof{k}$, and $u_4\sqfunof{k-1}$, all of which are all located in close proximity. The sequence $u_1$ then affects the cluster $u_2-u_5$, where it is needed in conjunction with the computation of $u_4$. Note that whenever there is a cluster of edges going out from a vertex, there is a corresponding dense block in the matrix $K_1$.}
    \label{fig:big-tree}
\end{figure*}

\subsection{LQR problems with a tree structure}\label{sec:more-general-lqr}

In this section we will describe how to extend the previous example to transportation networks with an underlying tree structure. We will present the resulting controller sparsity patterns for a few cases that illustrate how both the topology and direction of transportation links affect the sparsity in the factorisation of the control law. Note that the extensions here are the simplest possible, allowing many of the previously derived formulas from \cref{sec:simple-lqr} to be used without change. These results can be significantly generalised without compromising the resulting sparsity properties of the control laws. 

Consider again the transportation network illustrated in \Cref{fig:state-graph}, but suppose that the storage facilities are connected with an underlying tree topology. That is, the graph that characterises the transportation network topology (\Cref{fig:example-graph} for the setup in \cref{sec:simple-lqr}) is a tree graph rather than a line graph. Following \cref{sec:lqrsetup}, an \gls{lqr} problem can be associated with any such transportation network. Note that by appropriately updating the matrices $A$, $B$ and $C$ this problem takes the same form as in \cref{eq:simple-lqr}, and all the intermediate formulas can continue to be used, including the variable transform in \cref{eq:variable-transform} (it turns out that in this case $A = A^2$, meaning that \cref{eq:w} can also continue to be used without modification). 

Proceeding as in \cref{sec:lqrleastsquares} allows the \gls{lqr} problem to be converted into a least squares problem on the form in \cref{eq:ls-problem}. The resulting $\opfont{M}$ has sparsity pattern structured by the graph that describes the transportation network topology, just as in \Cref{fig:example-graph}. More specifically, the graph has one vertex for every storage facility, and one directed edge for each transportation link. The matrix $\opfont{M}$ then has one row for each vertex (storage facility), and one column for each edge in the graph (transportation link). Each outflow from a vertex is represented by an entry $-\opfont{1}$, and each inflow by an entry $r\shiftop^*$ modelling the unit delay associated with the transportation. 

The resulting least squares problems can then be solved as in \cref{sec:lqrbacksub,sec:lqrstructurecontrol}. We now present the sparsity patterns for the resulting structured control laws for a few special cases. The transportation network topologies and controller sparsity patterns are illustrated in \Cref{fig:big-tree,fig:small-examples}.

\subsubsection{A larger example}
To highlight the sparsity inherent in the factorisation of the control law in \cref{eq:control-law}, we study a fairly large transportation network consisting of 21 storage facilities arranged in a tree structure as depicted in the graph in \Cref{fig:big-tree}. The sparsity patterns of $K_1$ and $K_2$ are shown below the graph. Observe that the locations of the non-zero entries closely mirrors the local structure in the graph.

\subsubsection{The role of transportation direction}
An aspect that makes the control for the transportation network in \Cref{fig:big-tree} so sparse is that all transportation leaves a central storage facility corresponding to vertex $y_{21}$. The control law becomes slightly more dense when there are multiple transportation routes that can supply a storage facility. To get a feeling for how the direction impacts the sparsity, compare \Cref{fig:small-examples}a.) with \Cref{fig:small-examples}b.), as well as \Cref{fig:small-examples}c.) with the introductory example in \cref{sec:intro}.

\subsubsection{Control law without sparse spectral factor}\label{sec:tvspecfac}

As was established in \cref{sec:spec}, in some cases no sparsity preserving spectral factorisation exists. This is the case for the transportation network illustrated by the graph in \Cref{fig:small-examples}c.). We see that despite this a structured control law can be obtained by exploiting the sparsity properties of the operator factorisation. As in all the other examples, the gain $K_1^{-1} K_2$ is equal the classic \gls{lqr} state-feedback gain. Note that in this case \Cref{thm:control-law} cannot be used directly to compute $K_1$ and $K_2$. However it is in fact always possible to obtain a sparsity exploiting factorisation given any $\opfont{M}$ as outlined in this section. Describing the details precisely is tricky, so we omit them here, but the provided software implementation covers the general case.

\begin{figure}
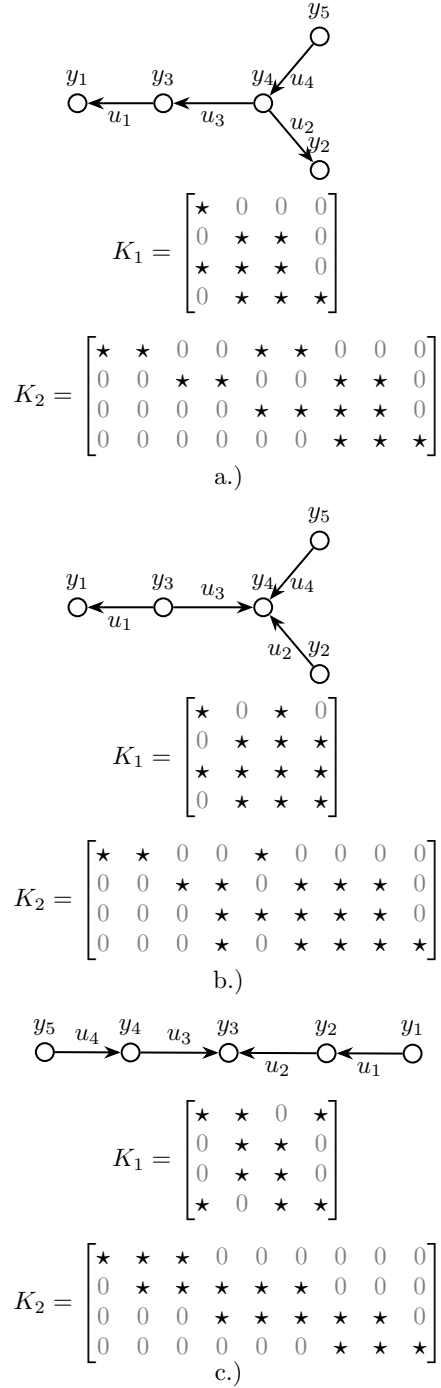

    \centering
    \scalebox{0.95}{
    \begin{tikzpicture}[
		every node/.append style={circle,draw=black, text=black,inner sep=0pt, minimum size=2.5mm}
	]

    \pgfmathsetmacro{\distGraph}{2.1}
    \pgfmathsetmacro{\distEq}{2}

	\node [label={$y_{1}$}]at (-2.115154e+00,4.440892e-16)(y1) { };
	\node [label={$y_{2}$}]at (1.273561e+00,-9.350289e-01)(y2) { };
	\node [label={$y_{3}$}]at (-9.151542e-01,2.775558e-16)(y3) { };
	\node [label={$y_{4}$}]at (4.831865e-01,-5.551115e-17)(y4) { };
	\node [label={$y_{5}$}]at (1.273561e+00,9.350289e-01)(y5) { };
	\path (y3) edge[-Stealth] node[auto, draw=none] {$u_{1}$} (y1);
	\path (y4) edge[-Stealth] node[auto, draw=none] {$u_{2}$} (y2);
	\path (y4) edge[-Stealth] node[auto, draw=none] {$u_{3}$} (y3);
	\path (y5) edge[-Stealth] node[auto, draw=none] {$u_{4}$} (y4);

    \node [draw=none] at (0, -\distGraph)(K1) {$\input{unidirectional-tree-K1.tex}$};
    \node [draw=none] at (0, -\distGraph-\distEq)(K2) {$\input{unidirectional-tree-K2.tex}$};
    \node [draw=none] at (0, -2.5*\distGraph)(T) {a.)};
\end{tikzpicture}}

    \vspace{-1.6cm}

    \scalebox{0.95}{
    \begin{tikzpicture}[
		every node/.append style={circle,draw=black, text=black,inner sep=0pt, minimum size=2.5mm}
	]

    \pgfmathsetmacro{\distGraph}{2.1}  
    \pgfmathsetmacro{\distEq}{2}  

	\node [label={$y_{1}$}]at (-2.115154e+00,4.440892e-16)(y1) { };
	\node [label={$y_{2}$}]at (1.273561e+00,-9.350289e-01)(y2) { };
	\node [label={$y_{3}$}]at (-9.151542e-01,2.775558e-16)(y3) { };
	\node [label={$y_{4}$}]at (4.831865e-01,-5.551115e-17)(y4) { };
	\node [label={$y_{5}$}]at (1.273561e+00,9.350289e-01)(y5) { };
	\path (y3) edge[-Stealth] node[auto, draw=none] {$u_{1}$} (y1);
	\path (y2) edge[-Stealth] node[auto, draw=none] {$u_{2}$} (y4);
	\path (y3) edge[-Stealth] node[auto, draw=none] {$u_{3}$} (y4);
	\path (y5) edge[-Stealth] node[auto, draw=none] {$u_{4}$} (y4);

    \node [draw=none] at (0, -\distGraph)(K1) {$\input{bidirectional-tree-K1.tex}$};
    \node [draw=none] at (0, -\distGraph-\distEq)(K2) {$\input{bidirectional-tree-K2.tex}$};
    \node [draw=none] at (0, -2.5*\distGraph)(T) {b.)};
\end{tikzpicture}}

    \vspace{-1.5cm}
    \scalebox{0.95}{
    \begin{tikzpicture}[
		every node/.append style={circle,draw=black, text=black,inner sep=0pt, minimum size=2.5mm}
	]

    \pgfmathsetmacro{\distGraph}{1.5}
    \pgfmathsetmacro{\distEq}{2}
        
	\node [label={$y_{1}$}]at (2.574058e+00,-1.072717e-02)(y1) { };
	\node [label={$y_{2}$}]at (1.370557e+00,-5.711694e-03)(y2) { };
	\node [label={$y_{3}$}]at (-9.408947e-05,3.714811e-07)(y3) { };
	\node [label={$y_{4}$}]at (-1.372265e+00,5.718807e-03)(y4) { };
	\node [label={$y_{5}$}]at (-2.572255e+00,1.071969e-02)(y5) { };
	\path (y1) edge[-Stealth] node[auto, draw=none] {$u_{1}$} (y2);
	\path (y2) edge[-Stealth] node[auto, draw=none] {$u_{2}$} (y3);
	\path (y4) edge[-Stealth] node[auto, draw=none] {$u_{3}$} (y3);
	\path (y5) edge[-Stealth] node[auto, draw=none] {$u_{4}$} (y4);

    \node [draw=none] at (0, -\distGraph)(K1) {$\input{tv-string-K1.tex}$};
    \node [draw=none] at (0, -\distGraph-\distEq)(K2) {$\input{tv-string-K2.tex}$};
    \node [draw=none] at (0, -3*\distGraph)(T) {c.)};

\end{tikzpicture}}

    \vspace{-2cm}

    \caption{Illustration of the effect of transportation direction on the sparsity of the control law in \cref{eq:control-law}. Just as in \Cref{fig:big-tree}, the graph illustrates the topology of the transportation network, and the $\boldsymbol{\star}$ symbol the non-zero entries in the control law \cref{eq:control-law}. Note that the graphs in a.) and b.) are the same except for the direction of the edges. The graph in b.) has multiple edges entering the same vertex. This makes the control law less sparse in general. The graph in c.) shows the same characteristic, however this time it comes from the fact that the underlying factorisation is not a spectral factor (no sparsity preserving spectral factor exists, as explained in \cref{sec:spec}).}
    \label{fig:small-examples}
\end{figure}

\section{Conclusions}\label{sec:conclusions}

In this paper we have developed a computational framework for performing  Cholesky factorisation for matrices
of shift operators. The tools allow for exact solution of least squares problems in a sparsity exploiting manner. We have given examples where the Cholesky algorithm produces spectral factors, however we also proved that there are matrices for which no sparsity preserving spectral factorisation exists. We use our results to reveal sparsity properties in the control laws of particular \gls{lqr} problems associated with transportation.

\section*{Acknowledgement}
The authors thank Anders Hansson for his support throughout the
project and in particular for the feedback he has given on an early draft
of this paper.

\bibliographystyle{unsrt}
\bibliography{references}

\appendix
\section{Triangular factorisations preserve sparsity}
The following lemma is used in the proof of \Cref{thm:specfac}. The lemma concerns $n\times{}n$ matrices, where each entry is an operator on a Hilbert space $\mathcal{H}$. It shows that lower-upper factorisations (such as the Cholesky factorisation) always produce factors with the same sparsity pattern, and further that these factors always satisfy a formula involving the diagonal entries of the factors.

\begin{lemma}\label{lem:triangular}
Let $\mathcal{H}$ be a Hilbert space. If $\opfont{L}_1:\mathcal{H}^n\rightarrow{}\mathcal{H}^n$ and  $\opfont{L}_2:\mathcal{H}^n\rightarrow{}\mathcal{H}^n$ are invertible and lower triangular, and
\begin{equation}\label{eq:sim}
\opfont{L}_1\opfont{L}_1^*=\opfont{L}_2\opfont{L}_2^*,
\end{equation}
then $\opfont{L}_2=\opfont{L}_1\mathrm{diag}\,\funof{\opfont{L}_1}^{-1}\mathrm{diag}\,\funof{\opfont{L}_2}$.
\end{lemma}
\begin{proof}
    First note that since $\opfont{L}_1$ and $\opfont{L}_2$ are invertible and lower triangular, the matrices of operators $\opfont{L}_1^{-1}\opfont{L}_2$ and $\opfont{L}_2^{-1}\opfont{L}_1$ are lower triangular. It follows from \cref{eq:sim} that
        \[
    \opfont{L}_1^{-1}\opfont{L}_2\funof{\opfont{L}_1^{-1}\opfont{L}_2}^*=\opfont{I}\;\;\text{and}\;\;
\opfont{L}_2^{-1}\opfont{L}_1\funof{\opfont{L}_2^{-1}\opfont{L}_1}^*=\opfont{I}.
    \]
    This implies that $\opfont{L}_1^{-1}\opfont{L}_2=\mathrm{diag}\,\funof{\opfont{L}_1^{-1}\opfont{L}_2}$. Finally since $\opfont{L}_1^{-1}$ and $\opfont{L}_2$ are lower triangular,
    \[
    \mathrm{diag}\,\funof{\opfont{L}_1^{-1}\opfont{L}_2}=\mathrm{diag}\,\funof{\opfont{L}_1}^{-1}\mathrm{diag}\,\funof{\opfont{L}_2},
    \]
    and the proof is complete. 
\end{proof}

\end{document}